\documentclass{mcom-l}

\usepackage{amssymb}

\usepackage{graphicx}

\newtheorem{theorem}{Theorem}[section]
\newtheorem{lemma}[theorem]{Lemma}

\theoremstyle{definition}

\theoremstyle{remark}
\newtheorem{remark}[theorem]{Remark}

\numberwithin{equation}{section}

\newcommand\fpp[2]{{\frac{\partial #1}{\partial#2}}}
\newcommand\dd{{\mathop{}\!\mathrm{d}}}

\newcommand\cper{{\mathcal{C}_{\mathrm{per}}}}
\newcommand\eperx{{\mathcal{E}_{\mathrm{per}}^\mathrm{x}}}
\newcommand\epery{{\mathcal{E}_{\mathrm{per}}^\mathrm{y}}}
\newcommand\eperz{{\mathcal{E}_{\mathrm{per}}^\mathrm{z}}}
\newcommand\veper{{\vec{\mathcal{E}}_{\mathrm{per}}}}
\newcommand\vf{{\vec{f}}}

\newtheorem{proposition}[theorem]{Proposition}

\usepackage{amsmath}
\usepackage{multirow}
\usepackage{algorithm}
\usepackage{amsrefs}
\usepackage{color}

\begin{document}

\title[An Iterative Solver for CH equation with Log-Singularities]
{Overcoming Logarithmic Singularities in the Cahn-Hilliard Equation with Flory-Huggins Potential: An Unconditionally Convergent ADMM Approach}


%    author one information
\author{Ruo Li}
\address{CAPT, LMAM and School of Mathematical Sciences,
Peking University, Beijing, China;
Chongqing Research Institute of Big Data, Peking University, Chongqing, China}
\curraddr{}
\email{rli@math.pku.edu.cn}
\thanks{The first author's research was partially supported by the National
Natural Science Foundation of China (Grant No. 12288101)}

%    author two information
\author{Shengtong Liang}
\address{School of Mathematical Sciences, Peking University, Beijing, China}
\curraddr{}
\email{liangshengtong@pku.edu.cn}
\thanks{}

%    author three information
\author{Zhonghua Qiao}
\address{Corresponding author. Department of Applied Mathematics and Research Institute for Smart
Energy, The Hong Kong Polytechnic University, Hung Hom, Kowloon, Hong Kong}
\curraddr{}
\email{zhonghua.qiao@polyu.edu.hk}
\thanks{The third author's research was partially
supported by the Hong Kong Research Grants Council (RFS project RFS2021-5S03
and GRF project 15305624) and the CAS AMSS-PolyU Joint Laboratory of Applied Mathematics (No. JLFS/P-501/24).}

\thanks{This work was partially supported by High-performance Computing Platform of Peking
University.}

\dedicatory{}

\date{}

%    \subjclass is required.


%    Abstract is required.
\begin{abstract}
  %The Cahn-Hilliard equation with the Flory-Huggins potential is a fundamental
%  phase field model for phase separation. Its solution $u$ is confined to
%  $(-1,1)$ due to the logarithmic singularities at $\pm 1$. Convex splitting
%  schemes preserve this bound and ensure unconditional unique solvability.
%  However, their implementation requires solving nonlinear systems with singular
%  logarithmic terms at each time step. Iterative convergence and maintaining the
%  bound during iterations are challenging. Existing iterative solvers for these
%  systems typically require restrictive assumptions to guarantee convergence,
%  such as the strict separation property ($\|u\|_\infty<1$) or a small time step
%  size, limiting their practical utility.  We overcome this challenge by
%  developing a new iterative solver specifically designed for these singular
%  nonlinear systems. Our solver is based on a novel alternating direction method
%  of multipliers (ADMM) framework, features a carefully designed variable
%  splitting that decouples the logarithmic nonlinearity, enabling efficient
%  handling of singularities. Crucially, we prove unconditional convergence of
%  this new ADMM approach, requiring no time step restrictions or strict
%  separation assumptions, which allows us to take full advantage of the convex
%  splitting method's unconditional solvability. Numerical experiments
%  demonstrate the effectiveness and robustness of the proposed iterative
%  process, validating both our algorithm and theoretical analysis.
The Cahn-Hilliard equation with Flory-Huggins potential serves as a fundamental
phase field model for describing phase separation phenomena. Due to the presence
of logarithmic singularities at  $u=\pm 1$, the solution $u$ is constrained
within the interval $(-1,1)$. While convex splitting schemes are commonly
employed to preserve this bound and guarantee unconditional unique solvability,
their practical implementation requires solving nonlinear systems containing
singular logarithmic terms at each time step. This introduces significant
challenges in both ensuring convergence of iterative solvers and maintaining the
solution bounds throughout the iterations. Existing solvers often rely on
restrictive conditions---such as the strict separation property or small time
step sizes---to ensure convergence, which can limit their applicability. In this
work, we introduce a novel iterative solver that is specifically designed for
singular nonlinear systems, with the use of a variant of the alternating
direction method of multipliers (ADMM). By developing a tailored variable
splitting strategy within the ADMM framework, our method efficiently decouples
the challenging logarithmic nonlinearity, enabling effective handling of
singularities. Crucially, we rigorously prove the unconditional convergence of
our ADMM-based solver, which removes the need for time step constraints or
strict separation conditions. This allows us to fully leverage the unconditional
solvability offered by convex splitting schemes. Comprehensive numerical
experiments demonstrate the superior efficiency and robustness of our ADMM
variant, strongly validating both our algorithmic design and theoretical
results.
\end{abstract}

\maketitle

\section{Introduction}

The Cahn-Hilliard equation \cite{cahn1958free} is a fundamental model within the
phase field methodology. It describes the diffusion dynamics at the interface
between two phases. As an $H^{-1}$ gradient flow of an energy functional, its
dynamics are governed by the choice of potential. In this paper, we focus on the
Flory-Huggins energy potential given by
\begin{equation}
    \label{eq:flory-huggins}
    E(u)=\int_\Omega \left((1+u)\log(1+u)+(1-u)\log(1-u)-\frac{\theta_0}{2}u^2
    +\frac{\epsilon^2}{2}|\nabla u|\right)\dd x.
\end{equation}
The Cahn-Hilliard equation with the Flory-Huggins potential can be written as
\begin{subequations}
\label{eq:CH}
\begin{align}
    &\fpp{u}{t}=\Delta w,\\
    &w=-\epsilon^2\Delta u+\log(1+u)-\log(1-u)-\theta_0 u,\quad
    x\in\Omega,\ t>0,
\end{align}
\end{subequations}
where $\Omega=(0,L)^d,d=2,3$ for a positive number $L$, and periodic boundary
conditions are imposed. $u=u(x,t)$ is the concentration of material components
in a two-phase system, and $w=w(x,t)$ is an auxiliary variable. The parameters
$\epsilon$ and $\theta_0$ are two positive constants associated with the diffuse
interface width.

There exist a variety of well-established numerical methods for the spatial
discretization of the Cahn-Hilliard equation, including finite difference
methods, finite element methods, and spectral methods, see e.g.,
\cites{chen1997applications,elliott1992error,furihata2001stable}. For temporal
discretization, numerous efficient numerical techniques have also been
developed. For further details, the reader is referred to \cites{DU2020425,
DuQiao2021,Hochbruck_Ostermann_2010,effective,tang2020efficient}. The
Cahn-Hilliard equation equipped with the Flory-Huggins energy potential provides
greater physical relevance compared to the frequently used polynomial potential,
as the logarithmic form of the Flory-Huggins potential can be rigorously derived
from theoretical considerations (see \cite{doi2013soft}). Notably, the natural
singularity at $u=\pm 1$ ensures that the solution remains within the physical
bound $-1<u<1$ without the need for artificial constraints. Given the importance
of this equation, many analytical studies have emerged in recent years, see
e.g., \cites{abels2007convergence,cherfils2011cahn,debussche1995,li2021,
miranville2012on,mirancille2004robust}. The logarithmic terms present in the
Flory-Higgins potential pose significant challenges for numerical
discretization. Although the physical constraint $-1<u<1$ has been rigorously
established \cite{miranville2019cahn}, only a limited number of numerical
schemes explicitly ensure the preservation of this bound. Copetti et al.
\cite{copetti1992numerical} developed a finite element method with implicit
Euler time discretization that is capable of preserving solution bounds.
However, the guarantee of unique solvability for their approach necessitates
stringent restrictions on the time step size. To overcome this limitation and
ensure unconditional unique solvability, Chen et al. \cite{chen2019positivity}
developed two convex splitting finite difference schemes, offering either
first-order or second-order accuracy in time.

%Convex splitting methods, as a semi-implicit temporal discretization, generate
%nonlinear discrete systems requiring iterative solutions at each time step.
%These iterations become particularly challenging for logarithmic potentials due
%to their strong nonlinearity. While researchers have acknowledged this
%difficulty and proposed specialized solvers, prior works rarely address bound
%preserving during nonlinear iterations, which is a critical requirement for
%rigorous convergence analysis.
Convex splitting methods, which employ semi-implicit temporal discretization, 
result in nonlinear discrete systems that necessitate iterative solution
procedures  at each time step. The challenge of solving these systems is
particularly pronounced for logarithmic potentials due to their strong
nonlinearity. Although this difficulty is well-recognized in the literature and
has inspired the development of specialized solvers, previous studies seldom
address the issue of bound preservation during nonlinear iterations---a crucial
aspect for establishing rigorous convergence results. Two notable strategies
have emerged for tackling these discrete nonlinear systems
\cites{chen2019positivity,diegel2025convergence}. Chen et al.
\cite{chen2019positivity} introduced a multigrid solver that employs
linearization techniques to handle the nonlinearities inherent in convex
splitting schemes. In their approach, the nonlinear logarithmic term is
approximated linearly near its singularities (specifically, outside the interval
$[-1+\delta,1-\delta]$ for a small parameter $\delta$), effectively extending
the domain of the nonlinear term to the entire real line $\mathbb{R}$. This
extension allows the iterative solver to proceed even if the numerical solution
temporarily leaves the physical interval $(-1,1)$. During each iteration, the
nonlinear term is treated explicitly, and the resulting linear system is solved
efficiently using multigrid methods. Despite its empirical success, this
strategy has two significant drawbacks: (a) the linearized system accurately
reflects the original problem only when the solution strictly satisfies the
separation property ($\|u\|_\infty\leq 1-\delta$) for sufficiently small
$\delta$, and (b) there is no theoretical guarantee of convergence.
%Diegel et al. \cite{diegel2025convergence} developed a preconditioned steepest
%descent method that combines gradient descent with exact line search. The
%algorithm iteratively computes a descent direction by solving a linearized
%system, and determines the optimal step size. While preserving the logarithmic
%structure of the nonlinear term. This design inherently confines iterates to the
%interval $(-1,1)$ by the singularity of the Flory-Huggins potential, thereby
%avoiding artificial regularization. While the method achieves linear convergence
%under some assumptions, its applicability is severely constrained by three key
%limitations: (a) The convergence proof relies on the strict separation property,
%and there is no convergence analysis without this assumption. (b) Convergence
%requires sufficiently small time and spatial steps, contradicting the
%unconditional stability promised by convex splitting schemes. (c) The algorithm
%can only deal with the first-order numerical scheme while the second-order cases
%are not included.
Alternatively, Diegel et al. \cite{diegel2025convergence} proposed a
preconditioned steepest descent method that integrates gradient descent with an
exact line search. In this framework, each iteration involves computing a
descent direction by solving a linearized system, followed by determining the
optimal step size. Importantly, the method preserves the logarithmic structure
of the nonlinear term, ensuring that iterates remain confined to the interval
$(-1,1)$ due to the singularity of the Flory-Huggins potential, thus eliminating
the need for artificial regularization. While the method achieves linear
convergence under certain assumptions, its practical applicability is hindered 
by three main limitations: (a) the convergence analysis depends on the strict
separation property, with no results available in its absence; (b) convergence
is guaranteed only for sufficiently small temporal and spatial discretization
steps, which is inconsistent with the unconditional stability typically
associated with convex splitting schemes; and (c) the method is only applicable
to first-order numerical schemes, with second-order cases remaining unaddressed.

%To summarize, the aforementioned two methods share the following fundamental
%shortcomings: (a) reliance on the strict separation property whose validity in
%3D remains an open problem \cite{miranville2019cahn}, (b) lack of unconditional
%convergence independent of discrete step sizes, negating the unconditional
%solvability inherent to convex splitting schemes, and (c) limited extensibility
%to high-order numerical schemes or more general phase field models since the
%absence of a unified algorithmic framework suitable for theoretical analysis.

In summary, the two methods discussed above exhibit several key limitations: (a)
dependence on the strict separation property, the validity of which in three
dimensions is still unresolved \cite{miranville2019cahn}; (b) absence of
unconditional convergence that holds regardless of discrete step sizes, thereby
undermining the unconditional solvability typically provided by convex-splitting
schemes; and (c) restricted applicability to higher-order numerical methods or
broader classes of phase field models, due to the lack of a unified algorithmic
framework amenable to rigorous theoretical analysis.

In this paper, we propose a novel iterative solver for the nonlinear discrete
system that emerges from convex splitting finite difference discretizations of
the Cahn-Hilliard equation with the Flory-Huggins potential
\cite{chen2019positivity}. Our approach reformulates the nonlinear system as a
saddle-point minimax optimization problem characterized by a convex-concave
objective function. To solve this minimax problem, we develop an alternating
direction method of multipliers (ADMM) framework. Specifically, we decompose the
objective function into linear and nonlinear components, allowing us to address
each part separately in each iteration. The subproblem corresponding to the
linear component is efficiently solvable due to its constant coefficient matrix,
which enables the use of discrete fast Fourier transform (FFT) techniques. For
the nonlinear component, we reduce the problem to several independent scalar
equations, each possessing a monotonic derivative and thus amenable to efficient
solution via Newton's method. Extensive numerical experiments demonstrate the
robust convergence and effectiveness of our new iterative solver.

The proposed algorithm is straightforward to implement and systematically
overcomes the shortcomings of previous methods through three main innovations:
(a) Our convergence analysis does not rely on the strict separation property.
Moreover, by retaining the complete logarithmic form of the Flory-Huggins
potential, the solution is inherently restricted to the interval $(-1,1)$
throughout the iterative process. (b) The new iterative scheme is
unconditionally convergent, regardless of discrete time step sizes, thereby
preserving the unconditional unique solvability characteristic of convex
splitting approaches. This marks a substantial advancement over existing
techniques. (c) Our method is formulated within a unified ADMM-based framework,
which readily accommodates second-order temporal discretizations and
three-dimensional problems. Theoretical convergence guarantees are maintained
across all spatial dimensions and time discretizations, ensuring both robustness
and reliability. Additionally, this solver can be easily adapted to other phase
field models within the same ADMM framework.

The remainder of this paper is organized as follows: Section \ref{sec:review}
reviews finite difference discretizations of the Cahn-Hilliard equation using
convex splitting. Section \ref{sec:admm} introduces our novel iterative solver
for these discrete schemes. In Section \ref{sec:convergence}, we establish the
convergence of the solver and discuss additional properties. Numerical results
demonstrating the solver's effectiveness and robustness are presented in Section
\ref{sec:numerical}. Finally, concluding remarks are given in Section
\ref{sec:conclusion}.

\section{Review of the numerical schemes}
\label{sec:review}

In this section, we define the spatial discretization notations, and introduce
the first-order and second-order convex splitting numerical schemes for the
Cahn-Hilliard equation. The following notations are partially taken from
\cites{chen2019positivity,wise2009energy}.

We present the three-dimensional case here, and the two-dimensional case is
similar and therefore omitted. For simplicity, we use periodic boundary
conditions, that is, the domain $\Omega=(0,L)^3$, the periodic boundary
conditions for equation \eqref{eq:CH} are written as
\begin{equation}
    u(x+L,y,z)=u(x,y+L,z)=u(x,y,z+L)=u(x,y,z),\quad
    \forall (x,y,z)\in\mathbb{R}^3.
\end{equation}
For any $N\in\mathbb{N}_+$, the mesh size is given by $h=L/N$, and we use the
same mesh spacing in the $x$-direction and $y$-direction. Additionally, the
following two uniform grids are introduced,
\begin{equation}
\label{eq:grids}
    E=\{p_{i+\frac{1}{2}}|i\in\mathbb{Z}\},\quad C=\{p_i|i\in\mathbb{Z}\},
\end{equation}
where $p_i=p(i)=(i-\frac{1}{2})h$ with grid spacing $h>0$.

With the grids $C$ and $E$ in \eqref{eq:grids}, we define four $3$D discrete
$N^3$-periodic function spaces,
{\small \begin{subequations}
    \begin{align}
    &\cper=
    \{\nu:C\times C\times C\to\mathbb{R}|\nu_{i,j,k}
    =\nu_{i+\lambda N,j+\mu N,k+\theta N},\
    \forall i,j,k,\lambda,\mu,\theta\in\mathbb{Z}\},\\
    &\eperx=\{\nu:E\times C\times C\to\mathbb{R}|\nu_{i+\frac{1}{2},j,k}
    =\nu_{i+\frac{1}{2}+\lambda N,j+\mu N,k+\theta N},\
    \forall i,j,k,\lambda,\mu,\theta\in\mathbb{Z}\},\\
    &\epery=\{\nu:C\times E\times C\to\mathbb{R}|\nu_{i,j+\frac{1}{2},k}
    =\nu_{i+\lambda N,j+\frac{1}{2}+\mu N,k+\theta N},\
    \forall i,j,k,\lambda,\mu,\theta\in\mathbb{Z}\},\\
    &\eperz=\{\nu:C\times C\times E\to\mathbb{R}|\nu_{i,j,k+\frac{1}{2}}
    =\nu_{i+\lambda N,j+\mu N,k+\frac{1}{2}+\theta N},\
    \forall i,j,k,\lambda,\mu,\theta\in\mathbb{Z}\},
    \end{align}
\end{subequations}}%
where $\nu_{i,j,k}=\nu(p_i,p_j,p_k)$. Finally, the space
$\veper=\eperx\times\epery\times\eperz$ is introduced.

The spatial difference operators are defined as
\begin{subequations}
    \begin{align}
    (D_x\nu)_{i+\frac{1}{2},j,k}=
    \frac{1}{h}(\nu_{i+1,j,k}-\nu_{i,j,k}),\\
    (D_y\nu)_{i,j+\frac{1}{2},k}=
    \frac{1}{h}(\nu_{i,j+1,k}-\nu_{i,j,k}),\\
    (D_z\nu)_{i,j,k+\frac{1}{2}}=
    \frac{1}{h}(\nu_{i,j,k+1}-\nu_{i,j,k}),
    \end{align}
\end{subequations}
with $D_x:\cper\to\eperx$,  $D_y:\cper\to\epery$, and $D_z:\cper\to\eperz$.
Similarly,
\begin{subequations}
    \begin{align}
    (d_x\nu)_{i,j,k}=
    \frac{1}{h}(\nu_{i+\frac{1}{2},j,k}-\nu_{i-\frac{1}{2},j,k}),\\
    (d_y\nu)_{i,j,k}=
    \frac{1}{h}(\nu_{i,j+\frac{1}{2},k}-\nu_{i,j-\frac{1}{2},k}),\\
    (d_z\nu)_{i,j,k}=
    \frac{1}{h}(\nu_{i,j,k+\frac{1}{2}}-\nu_{i,j,k-\frac{1}{2}}),
    \end{align}
\end{subequations}
with $d_x:\eperx\to\cper$, $d_y:\epery\to\cper$, and $d_z:\eperz\to\cper$. The
discrete gradient $\nabla_h:\cper\to\veper$ is defined as
\begin{equation}
    (\nabla_h\nu)_{i,j,k}
    =((D_x\nu)_{i+\frac{1}{2},j,k},(D_y\nu)_{i,j+\frac{1}{2},k},
    (D_z\nu)_{i,j,k+\frac{1}{2}}),
\end{equation}
and the discrete divergence $\nabla_h\cdot:\veper\to\cper$ becomes
\begin{equation}
    (\nabla_h\cdot\vf)_{i,j,k}=
    (d_xf^x)_{i,j,k}+(d_yf^y)_{i,j,k}+(d_zf^z)_{i,j,k},
\end{equation}
where $\vf=(f^x,f^y,f^z)\in\veper$. The $3$D discrete Laplacian,
$\Delta_h:\cper\to\cper$, is given by
\begin{equation}
\begin{aligned}
    &(\Delta_h\nu)_{i,j,k}=(\nabla_h\cdot\nabla_h \nu)_{i,j,k}\\
    &\ \ =\frac{1}{h^2}(\nu_{i+1,j,k}+\nu_{i-1,j,k}+\nu_{i,j+1,k}+\nu_{i,j-1,k}
    +\nu_{i,j,k+1}+\nu_{i,j,k-1}-6\nu_{i,j,k}).
\end{aligned}
\end{equation}

We define the following grid inner products:
\begin{subequations}
    \begin{align}
    &\langle \nu,\xi\rangle
    =h^2\sum_{i,j,k=1}^N\nu_{i,j,k}\xi_{i,j,k},\quad\nu,\xi\in\cper,\\
    &\langle \nu,\xi\rangle
    =h^2\sum_{i,j,k=1}^N\nu_{i+\frac{1}{2},j,k}\xi_{i+\frac{1}{2},j,k},\quad
    \nu,\xi\in\eperx,\\
    &\langle \nu,\xi\rangle
    =h^2\sum_{i,j,k=1}^N\nu_{i,j+\frac{1}{2},k}\xi_{i,j+\frac{1}{2},k},\quad
    \nu,\xi\in\epery,\\
    &\langle \nu,\xi\rangle
    =h^2\sum_{i,j,k=1}^N\nu_{i,j,k+\frac{1}{2}}\xi_{i,j,k+\frac{1}{2}},\quad
    \nu,\xi\in\eperz,\\
    &\langle \vf_1,\vf_2\rangle
    =\langle f_1^x,f_2^x\rangle+\langle f_1^y,f_2^y\rangle
    +\langle f_1^z,f_2^z\rangle,\ \
    \vf_i=(f_i^x,f_i^y,f_i^z)\in\veper,\ i=1,2.
    \end{align}
\end{subequations}

The norms for cell-centered functions are accordingly introduced. If
$\nu\in\cper$, then $\|\nu\|_2^2=\langle\nu,\nu\rangle$,
$\|\nu\|_1=\langle |\nu|,1\rangle$, and
$\|\nu\|_\infty=\max_{1\leq i,j\leq N}|\nu_{i,j}|$. The gradient norms are
similarly defined: for $\nu\in\cper$,
$\|\nabla_h\nu\|_2^2=\langle\nabla_h\nu,\nabla_h\nu\rangle$.

\begin{proposition}
    For any $\psi,\nu\in\cper$ and any $\vf\in\veper$, the following summation
    by parts formulas are valid:
    \begin{equation}
        \langle\psi,\nabla_h\cdot\vf\rangle
        =-\langle\nabla_h\psi,\vf\rangle,\quad
        \langle\psi,\Delta_h\nu\rangle=-\langle\nabla_h\psi,\nabla_h\nu\rangle.
    \end{equation}
\end{proposition}

Now we can introduce the first-order and second-order convex splitting schemes
for the Cahn-Hilliard equation \eqref{eq:CH}. Consider a uniform partition of
time, $0=t_0<t_1<\cdots<t_F=T$, such that $t_n=n\tau$. Let $u^n$ denote the
numerical solution at time $t_n$. The first-order convex splitting scheme to the
Cahn-Hilliard equation \eqref{eq:CH}, proposed in \cite{chen2019positivity}, is
stated as follows: given $u^n\in\cper$, find $u^{n+1}\in\cper$ such that
\begin{subequations}
\label{eq:first-order}
    \begin{align}
    &u^{n+1}-u^n=\tau\Delta_h w^{n+1},\\
    &w^{n+1}
    =\log(1+u^{n+1})-\log(1-u^{n+1})-\theta_0 u^n-\epsilon^2\Delta_h u^{n+1}.
    \end{align}
\end{subequations}
The second-order convex splitting scheme for the Cahn-Hilliard equation
\eqref{eq:CH}, also proposed in \cite{chen2019positivity}, is stated as follows:
given $u^n,u^{n-1}\in\cper$, find $u^{n+1}\in\cper$ such that
\begin{subequations}
\label{eq:second-order}
    \begin{align}
    &3u^{n+1}-4u^n+u^{n-1}=2\tau\Delta_h w^{n+1},\\
    &\begin{aligned}&w^{n+1}=\log(1+u^{n+1})-\log(1-u^{n+1})\\
    &\hspace{60pt}-\theta_0 (2u^n-u^{n-1})-\epsilon^2\Delta_h u^{n+1}
    -A\tau\theta_0^2\Delta_h(u^{n+1}-u^n).\end{aligned}
    \end{align}
\end{subequations}
A modified energy stability is available for the second-order scheme
\eqref{eq:second-order}, provided that the stabilization parameter
$A\geq 1/16$, see \cite{chen2019positivity} for more details. 
\section{The new iterative solvers based on ADMM framework}
\label{sec:admm}

In this section, we introduce new iterative solvers based on a novel variant of
the ADMM for both first-order and second-order schemes. The classic ADMM has
been extensively applied to minimization problems (see
\cites{bertsekas1997parallel, fazel2013hankel}). By adopting a similar
methodology, we derive an iterative solver tailored to minimax optimization
problems. Although our approach also targets minimax problems, the framework
developed here is fundamentally distinct from those proposed in
\cites{karabag2021deception, karabag2022alternating}.

\subsection{Iterative solver for the first-order convex splitting scheme}

First, we transform the first-order convex splitting scheme
\eqref{eq:first-order} into a minimax optimization problem.

For a given $u^n\in\cper$, define a discrete energy
\begin{equation}
\label{eq:firstZ}
    \begin{aligned}&Z(u,w)
    =\frac{\epsilon^2}{2}\|\nabla_h u\|_2^2+\langle 1+u,\log(1+u)\rangle
    +\langle 1-u,\log(1-u)\rangle\\
    &\hspace{100pt}-\theta_0 \langle u^n,u\rangle-\langle u-u^n,w\rangle
    -\frac{\tau}{2}\|\nabla_h w\|_2^2,\end{aligned}
\end{equation}
over the set $\cper\times\cper$. Then, solving \eqref{eq:firstZ} is
equivalent to finding the solution of the optimization problem
\begin{equation}
\label{eq:optini}
    \min_u\max_w\quad Z(u,w),\qquad \mathrm{s.t.}\quad u,w\in\cper.
\end{equation}

Dividing $Z(u,w)$ into the sum of two parts $Z_1(u,w)$ and $Z_2(u,w)$,
\begin{subequations}
    \begin{align}
    &Z_1(u,w)=\frac{\epsilon^2}{2}\|\nabla_h u\|_2^2
    -\alpha\langle u-u^n,w\rangle-\frac{\tau}{2}\|\nabla_h w\|_2^2,\\
    &\begin{aligned}
    &Z_2(u,w)=\langle 1+u,\log(1+u)\rangle+\langle 1-u,\log(1-u)\rangle\\
    &\hspace{100pt}
    -\theta_0\langle u^n,u\rangle-(1-\alpha)\langle u-u^n,w\rangle,
    \end{aligned}
    \end{align}
\end{subequations}
the optimization problem \eqref{eq:optini} can be equivalently converted to the
optimization problem
\begin{equation}
\label{eq:opt}
    \min_{u_1,u_2}\max_{w_1,w_2}\ Z_1(u_1,w_1)+Z_2(u_2,w_2),\quad
    u_1=u_2,\ w_1=w_2.
\end{equation}
We write the Lagrange function as
{\small \begin{equation}
\label{eq:largrange}
    \mathcal{L}(u_1,w_1,u_2,w_2,u_3,w_3)=Z_1(u_1,w_1)+Z_2(u_2,w_2)
    +\langle u_3,u_1-u_2\rangle-\langle w_3,w_1-w_2\rangle,
\end{equation}}%
where $u_3,w_3\in\cper$ are the Lagrange multipliers.
The augmented Lagrangian function of the problem \eqref{eq:opt} is
\begin{equation}
    \begin{aligned}
    &\mathcal{L}_{\rho_u,\rho_w}(u_1,w_1,u_2,w_2,u_3,w_3)\\
    &\hspace{20pt}
    =\mathcal{L}(u_1,w_1,u_2,w_2,u_3,w_3)+\frac{\rho_u}{2}\|u_1-u_2\|_2^2
    -\frac{\rho_w}{2}\|w_1-w_2\|_2^2,
    \end{aligned}
\end{equation}
where $\rho_u,\rho_w$ are two positive constants.

The $k$-th step of ADMM iteration for the optimization problem \eqref{eq:opt} is
{\small \begin{subequations}
    \begin{align}
    &(u_1^{(k+1)},w_1^{(k+1)})=\mathop{\arg\min\max}\limits_{(u_1,w_1)}
    \mathcal{L}_{\rho_u,\rho_w}
    (u_1,w_1,u_2^{(k)},w_2^{(k)},u_3^{(k)},w_3^{(k)}),
    \label{eq:admm_detail_1} \\
    &(u_2^{(k+1)},w_2^{(k+1)})=\mathop{\arg\min\max}\limits_{(u_2,w_2)}
    \mathcal{L}_{\rho_u,\rho_w}
    (u_1^{(k+1)},w_1^{(k+1)},u_2,w_2,u_3^{(k)},w_3^{(k)}),
    \label{eq:admm_detail_2} \\
    &u_3^{(k+1)}=u_3^{(k)}+\rho_u(u_1^{(k+1)}-u_2^{(k+1)}),\quad
    w_3^{(k+1)}=w_3^{(k)}+\rho_w(w_1^{(k+1)}-w_2^{(k+1)}),
    \end{align}
\end{subequations}}%
where $\mathop{\arg\min\max}_{(u_i,w_i)}$ means the optimal point for
corresponding minimax subproblem, i.e., $(u_i,w_i)$ is the optimal solution of
the subproblem
\begin{equation}
    \min_{u_i\in\cper}\max_{w_i\in\cper}\mathcal{L}_{\rho_u,\rho_w}(\cdots),
    \quad i=1,2,
\end{equation}
where the ellipse represents the corresponding variables in
\eqref{eq:admm_detail_1} and \eqref{eq:admm_detail_2}.

The detailed process of the iterative solver for the first-order convex
splitting method is presented in Algorithm \ref{alg:first-order}.

\begin{algorithm}[!t]
    \caption{Iterative solver for the first-order convex splitting scheme}
\label{alg:first-order}
    \raggedright
    1. Define $u_2^{(0)}=u^n$, $w_2^{(0)}=0$, $u_3^{(0)}=0$, and
    $w_3^{(0)}=0$.\\
    2. Select $\alpha\in(0,1)$, $\rho_u\in(0,+\infty)$,
    and $\rho_w\in(0,+\infty)$.\\
    3. For $k\in\mathbb{N}$, solve the linear system
    \begin{subequations}
    \label{eq:linear}
        \begin{align}
        &-\epsilon^2\Delta_hu_1^{(k+1)}-\alpha w_1^{(k+1)}+u_3^{(k)}
        +\rho_u(u_1^{(k+1)}-u_2^{(k)})=0,\\
        &\alpha(-u_1^{(k+1)}+u^n)+\tau\Delta_h w_1^{(k+1)}-w_3^{(k)}
        -\rho_w(w_1^{(k+1)}-w_2^{(k)})=0, \label{eq:linear_2}
        \end{align}
    \end{subequations}
    \hspace{10pt} for $u_1^{(k+1)},w_1^{(k+1)}\in\cper$.\\
    4. Solve the nonlinear system
    \begin{subequations}
    \label{eq:nonlinear}
        \begin{align}
        &\begin{aligned}
        &\log(1+u_2^{(k+1)})-\log(1-u_2^{(k+1)})-\theta_0u^n\\
        &\hspace{100pt}-(1-\alpha) w_2^{(k+1)}
        -u_3^{(k)}-\rho_u(u_1^{(k+1)}-u_2^{(k+1)})=0,\end{aligned}\\
        &(1-\alpha)(-u_2^{(k+1)}+u^n)+w_3^{(k)}
        +\rho_w(w_1^{(k+1)}-w_2^{(k+1)})=0,
        \label{eq:nonlinear_2}
        \end{align}
    \end{subequations}
    \hspace{10pt} for $u_2^{(k+1)},w_2^{(k+1)}\in\cper$.\\
    5. Update the Lagrange multipliers $u_3^{(k+1)},w_3^{(k+1)}\in\cper$ as
    \begin{equation}
    \label{eq:update}
        u_3^{(k+1)}=u_3^{(k)}+\rho_u(u_1^{(k+1)}-u_2^{(k+1)}),\quad
        w_3^{(k+1)}=w_3^{(k)}+\rho_w(w_1^{(k+1)}-w_2^{(k+1)}).
    \end{equation}
\end{algorithm}

\subsection{Iterative solver for the second-order convex splitting scheme}
Next, we transform the second-order convex splitting scheme
\eqref{eq:second-order} into a minimax optimization problem.

Let $u^n,u^{n-1}\in\cper$ be given. Define the discrete energy
{\small \begin{equation}
\label{eq:secondZ}
    \begin{aligned}
    &Z(u,w)=\left(\frac{\epsilon^2}{2}+\frac{A\tau\theta_0^2}{2}\right)
    \|\nabla_h u\|_2^2+\langle 1+u,\log(1+u)\rangle
    +\langle 1-u,\log(1-u)\rangle\\
    &\hspace{20pt}-\theta_0 \langle 2u^n-u^{n-1},u\rangle-\left\langle u
    -\frac{4}{3}u^n+\frac{1}{3}u^{n-1},w\right\rangle
    -\frac{\tau}{3}\|\nabla_h w\|_2^2,
    \end{aligned}
\end{equation}}%
over the admissible set $\cper\times\cper$. Then, solving \eqref{eq:secondZ}
is equivalent to finding the solution of the optimization problem
\eqref{eq:optini}.

The optimization problem \eqref{eq:optini} can be similarly transformed and
solved using the iterative solver. The specific algorithm for the
solver for the second-order convex splitting method is presented in Algorithm
\ref{alg:second-order}.

\begin{algorithm}[!t]
    \caption{Iterative solver for the second-order convex splitting scheme}
    \label{alg:second-order}
    \raggedright
    1. Define $u_2^{(0)}=u^n$, $w_2^{(0)}=0$, $u_3^{(0)}=0$, and
    $w_3^{(0)}=0$.\\
    2. Select $\alpha\in(0,1)$, $\rho_u\in(0,+\infty)$, and
    $\rho_w\in(0,+\infty)$.\\
    3. For $k\in\mathbb{N}$, solve the linear system
    \begin{subequations}
    \label{eq:secline}
        \begin{align}
        &(-\epsilon^2-A\tau\theta_0^2)\Delta_hu_1^{(k+1)}
        +A\tau\theta_0^2\Delta_h u^n
        -\alpha w_1^{(k+1)}+u_3^{(k)}+\rho_u(u_1^{(k+1)}-u_2^{(k)})=0,\\
        &\alpha\left(-u_1^{(k+1)}+\frac{4}{3}u^n-\frac{1}{3}u^{n-1}\right)
        +\frac{2\tau}{3}\Delta_h w_1^{(k+1)}-w_3^{(k)}
        -\rho_w(w_1^{(k+1)}-w_2^{(k)})=0,
        \end{align}
    \end{subequations}
    \hspace{10pt} for $u_1^{(k+1)},w_1^{(k+1)}\in\cper$.\\
    4. Solve the nonlinear system
    \begin{subequations}
    \label{eq:secnonlin}
        \begin{align}
        &\begin{aligned}&\log(1+u_2^{(k+1)})-\log(1-u_2^{(k+1)})
        -\theta_0(2u^n-u^{n-1})\\
        &\hspace{100pt}-(1-\alpha)w_2^{(k+1)}-u_3^{(k)}
        -\rho_u(u_1^{(k+1)}-u_2^{(k+1)})=0,\end{aligned}\\
        &(1-\alpha)\left(-u_2^{(k+1)}+\frac{4}{3}u^n-\frac{1}{3}u^{n-1}\right)
        +w_3^{(k)}+\rho_w(w_1^{(k+1)}-w_2^{(k+1)})=0,
        \end{align}
    \end{subequations}
    \hspace{10pt} for $u_2^{(k+1)},w_2^{(k+1)}\in\cper$.\\
    5. Update the Lagrange multipliers $u_3^{(k+1)},w_3^{(k+1)}\in\cper$ as
    \begin{equation}
        u_3^{(k+1)}=u_3^{(k)}+\rho_u(u_1^{(k+1)}-u_2^{(k+1)}),\quad
        w_3^{(k+1)}=w_3^{(k)}+\rho_w(w_1^{(k+1)}-w_2^{(k+1)}).
    \end{equation}
\end{algorithm}

\begin{remark}
    Equation \eqref{eq:linear} and \eqref{eq:secline} can be exactly and
    efficiently implemented by an FFT-based finite difference solver. Equation
    \eqref{eq:nonlinear} and \eqref{eq:secnonlin} can be decomposed into several
    scalar equations, each of which has a monotonic derivative so it can be
    solved efficiently by Newton's method.
\end{remark}

\section{The unconditional convergence of the solvers}
\label{sec:convergence}

We show the properties of the new solvers in this section, including
convergence, bound preserving, and mass conservation. First, we consider the
convergence. Here, we only consider the first-order scheme; the proofs for the
second-order scheme are similar, and thus we omit it here. We introduce some
lemmas first. The first lemma states the existence of the solution of the
optimization problem \eqref{eq:firstZ}.

\begin{lemma}
There exists a stationary point of the Lagrange function $\mathcal{L}$ defined
in \eqref{eq:largrange}.
\end{lemma}

\begin{proof}
    By the discussion in \cite{chen2019positivity}, for any $u^n\in\cper$ with
    $-1<u^n<1$ at a point-wise level, there exists a unique solution of the
    system \eqref{eq:first-order}, denoted as $(u^*,w^*)\in\cper\times\cper$
    with $-1<u^*<1$ at a point-wise level. That is,
    \begin{subequations}
        \begin{align}
        &u^*-u^n=\tau\Delta_h w^*,\\
        &w^*=-\epsilon^2\Delta_h u^* +\log(1+u^*)-\log(1-u^*)-\theta_0 u^n.
        \end{align}
    \end{subequations}
    Define $u_1^*,w_1^*,u_2^*,w_2^*,u_3^*,w_3^*\in\cper$ as
    \begin{subequations}
        \begin{align}
        &u_1^*=u^*,\quad w_1^*=w^*,\quad u_2^*=u^*,\quad w_2^*=w^*,\\
        &u_3^*=\epsilon^2\Delta_h u^*+\alpha w^*=\log(1+u^*)-\log(1-u^*)
        -\theta_0 u^*-(1-\alpha) w^*,\\
        &w_3^*=\alpha(-u^*+u^n)+\tau\Delta_h w^*=-(1-\alpha)(-u^*+u^n).
        \end{align}
    \end{subequations}
    Then, it holds that
    \begin{subequations}
        \begin{align}
        &-\epsilon^2\Delta_h u_1^*-\alpha w_1^*+u_3^*=0,\\
        &\alpha(-u_1^*+u^n)+\tau\Delta_h w_1^*-w_3^*=0,\\
        &\log(1+u_2^*)-\log(1-u_2^*)-\theta_0u^n-(1-\alpha)w_2^*-u_3^*=0,\\
        &(1-\alpha)(-u_2^*+u^n)+w_3^*=0,\\
        &u_1^*-u_2^*=0,\\
        &-w_1^*+w_2^*=0,
        \end{align}
    \end{subequations}
    which means $(u_1^*,w_1^*,u_2^*,w_2^*,u_3^*,w_3^*)$ is the stationary point
    of the function $\mathcal{L}$.
\end{proof}

In the following parts, for a fixed $u^n$, let $(u^*,w^*)$ be the solution of
the system \eqref{eq:first-order},
$\{(u_1^{(k)},w_1^{(k)},u_2^{(k)},w_2^{(k)},u_3^{(k)},w_3^{(k)})\}$ be the
sequence generated by Algorithm \ref{alg:first-order}, and
$(u_1^*,w_1^*,u_2^*,w_2^*,u_3^*,w_3^*)$ be the stationary point of $\mathcal{L}$
defined in \eqref{eq:largrange}. We define the sequence
\begin{equation}
\label{eq:error}
    \begin{aligned}
    &(e_{u1}^{(k)},e_{w1}^{(k)},e_{u2}^{(k)},e_{w2}^{(k)},e_{u3}^{(k)},
    e_{w3}^{(k)})\\
    &\hspace{20pt}=(u_1^{(k)},w_1^{(k)},u_2^{(k)},w_2^{(k)},u_3^{(k)},w_3^{(k)})
    -(u_1^*,w_1^*,u_2^*,w_2^*,u_3^*,w_3^*),\quad k\in\mathbb{N},
    \end{aligned}
\end{equation}
and the sequence $\{\Psi^{(k)}\}$ as
\begin{equation}
\label{eq:Psi}
    \Psi^{(k)}=\frac{1}{\rho_u}\|e_{u3}^{(k)}\|_2^2
    +\frac{1}{\rho_w}\|e_{w3}^{(k)}\|_2^2+\rho_u\|e_{u2}^{(k)}\|_2^2
    +\rho_w\|e_{w2}^{(k)}\|_2^2,\quad k\in\mathbb{N}.
\end{equation}
The following convergence analysis is similar to that of the classical ADMM
\cites{karabag2021deception, karabag2022alternating}.

\begin{lemma}
\label{lem:Psi}
    With the notations in \eqref{eq:error} and \eqref{eq:Psi}, it holds that,
    for $k\in\mathbb{N}$,
    \begin{equation}
        \begin{aligned}
        &\Psi^{(k)}-\Psi^{(k+1)}\geq
        \rho_u\|e_{u2}^{(k)}-e_{u2}^{(k+1)}\|_2^2
        +\rho_w\|e_{w2}^{(k)}-e_{w2}^{(k+1)}\|_2^2\\
        &\hspace{100pt} +\rho_u\|e_{u1}^{(k+1)}-e_{u2}^{(k+1)}\|_2^2
        +\rho_w\|e_{w1}^{(k+1)}-e_{w2}^{(k+1)}\|_2^2.
        \end{aligned}
    \end{equation}
\end{lemma}

\begin{proof}
    Since $-\Delta_h$ is positive semi-definite, we have
    \begin{equation}
    \label{eq:convex_1}
        \begin{aligned}
        0&\leq \langle -\epsilon^2\Delta_h(u_1^{(k+1)}-u_1^*),
        u_1^{(k+1)}-u_1^*\rangle\\
        &=\langle \alpha w_1^{(k+1)}-u_3^{(k)}-\rho_u(u_1^{(k+1)}-u_2^{(k)})
        -\alpha w_1^*-u_3^*, u_1^{(k+1)}-u_1^*\rangle\\
        &=\langle \alpha w_1^{(k+1)}-u_3^{(k+1)}+\rho_u(u_2^{(k)}-u_2^{(k+1)})
        -\alpha w_1^*-u_3^*, u_1^{(k+1)}-u_1^*\rangle\\
        &=\langle \alpha e_{w1}^{(k+1)}-e_{u3}^{(k+1)}
        +\rho_u(e_{u2}^{(k)}-e_{u2}^{(k+1)}), e_{u1}^{(k+1)}\rangle\\
        &=\langle -e_{u3}^{(k+1)}+\rho_u(e_{u2}^{(k)}-e_{u2}^{(k+1)}),
        e_{u1}^{(k+1)}\rangle+\alpha\langle e_{w1}^{(k+1)}, e_{u1}^{(k+1)}\rangle,
        \end{aligned}
    \end{equation}
    and
    \begin{equation}
    \label{eq:convex_2}
        \begin{aligned}
        0&\leq\langle -\tau\Delta_h (w_1^{(k+1)}-w_1^*),
        w_1^{(k+1)}-w_1^*\rangle\\
        &=\langle \alpha(-u_1^{(k+1)}+u^n)-w_3^{(k)}
        -\rho_w(w_1^{(k+1)}-w_2^{(k)})\\
        &\quad -\alpha(-u_1^*+u^n)+w_3^*,w_1^{(k+1)}-w_1^*\rangle\\
        &=\langle -\alpha u_1^{(k+1)}-w_3^{(k+1)}+\rho_w(w_2^{(k)}-w_2^{(k+1)})
        +\alpha u_1^*+w_3^*, w_1^{(k+1)}-w_1^*\rangle\\
        &=\langle -\alpha e_{u1}^{(k+1)}-e_{w3}^{(k+1)}
        +\rho_w(e_{w2}^{(k)}-e_{w2}^{(k+1)}), e_{w1}^{(k+1)}\rangle\\
        &=\langle -e_{w3}^{(k+1)}+\rho_w(e_{w2}^{(k)}-e_{w2}^{(k+1)}),
        e_{w1}^{(k+1)}\rangle-\alpha \langle  e_{u1}^{(k+1)}, e_{w1}^{(k+1)}
        \rangle.
        \end{aligned}
    \end{equation}

    By the property of the logarithmic function, we obtain that
    \begin{equation}
    \label{eq:convex}
        \begin{aligned}
        0&\leq \langle(\log(1+u_2^{(k+1)})-\log(1-u_2^{(k+1)}))\\
        &\quad-(\log(1+u_2^*)
        -\log(1-u_2^*)),u_2^{(k+1)}-u_2^* \rangle\\
        &\begin{aligned}
        &=\langle(\theta_0 u^n+(1-\alpha)w_2^{(k+1)}+u_3^{(k)}
        +\rho_u(u_1^{(k+1)}-u_2^{(k+1)}))\\
        &\quad -(\theta_0 u^n+(1-\alpha)w_2^*+u_3^*),
        u_2^{(k+1)}-u_2^*\rangle\end{aligned}\\
        &=\langle (1-\alpha)w_2^{(k+1)}+u_3^{(k+1)}
        -(1-\alpha)w_2^*-u_3^*,u_2^{(k+1)}-u_2^*\rangle\\
        &=\langle (1-\alpha)e_{w2}^{(k+1)}+e_{u3}^{(k+1)},e_{u2}^{(k+1)}\rangle\\
        &=\langle e_{u3}^{(k+1)},e_{u2}^{(k+1)}\rangle
        +(1-\alpha)\langle e_{w2}^{(k+1)},e_{u2}^{(k+1)}\rangle,
        \end{aligned}
    \end{equation}
    and
    {\small \begin{equation}
    \label{eq:convex_3}
        \begin{aligned}
        0&\leq \langle(\log(1+u_2^{(k)})-\log(1-u_2^{(k)}))\\
        &\hspace{20pt}-(\log(1+u_2^{(k+1)})
        -\log(1-u_2^{(k+1)})),u_2^{(k)}-u_2^{(k+1)} \rangle\\
        &=\langle(\theta_0 u^n+(1-\alpha)w_2^{(k)}+u_3^{(k-1)}+\rho_u(u_1^{(k)}-u_2^{(k)}))\\
        &\hspace{10pt}
        -(\theta_0 u^n+(1-\alpha)w_2^{(k+1)}+u_3^{(k)}+\rho_u(u_1^{(k+1)}-u_2^{(k+1)})),
        u_2^{(k)}-u_2^{(k+1)}\rangle\\
        &=\langle (1-\alpha)w_2^{(k)}+u_3^{(k)}-(1-\alpha)w_2^{(k+1)}-u_3^{(k)}\\
        &\hspace{20pt}-\rho_u(u_1^{(k+1)}-u_2^{(k+1)}),u_2^{(k)}-u_2^{(k+1)}\rangle\\
        &=\langle (1-\alpha)(e_{w2}^{(k)}-e_{w2}^{(k+1)})
        -\rho_u(e_{u1}^{(k+1)}-e_{u2}^{(k+1)}),e_{u2}^{(k)}-e_{u2}^{(k+1)}\rangle\\
        &=-\rho_u\langle e_{u1}^{(k+1)}-e_{u2}^{(k+1)},e_{u2}^{(k)}-e_{u2}^{(k+1)}\rangle
        +(1-\alpha)\langle e_{w2}^{(k)}-e_{w2}^{(k+1)},e_{u2}^{(k)}-e_{u2}^{(k+1)}\rangle.
        \end{aligned}
    \end{equation}}%
    We can also get that
    \begin{equation}
    \label{eq:convex_4}
        \begin{aligned}
        0&=\langle 0,w_2^{(k+1)}-w_2^*\rangle\\
        &=\langle ((1-\alpha)(-u_2^{(k+1)}+u^n)+w_3^{(k)}
        +\rho_w(w_1^{(k+1)}-w_2^{(k+1)}))\\
        &\hspace{100pt}-((1-\alpha)(-u_2^*+u^n)+w_3^*),w_2^{(k+1)}-w_2^*\rangle\\
        &=\langle -(1-\alpha)u_2^{(k+1)}+w_3^{(k+1)}+(1-\alpha)u_2^*-w_3^*,w_2^{(k+1)}
        -w_2^*\rangle\\
        &=\langle -(1-\alpha)e_{u2}^{(k+1)}+e_{w3}^{(k+1)}, e_{w2}^{(k+1)}\rangle\\
        &=\langle e_{w3}^{(k+1)}, e_{w2}^{(k+1)}\rangle
        -(1-\alpha)\langle e_{u2}^{(k+1)},e_{w2}^{(k+1)}\rangle,
        \end{aligned}
    \end{equation}
    and
    {\small \begin{equation}
    \label{eq:convex_5}
        \begin{aligned}
        0&=\langle 0,w_2^{(k)}-w_2^{(k+1)}\rangle\\
        &=\langle ((1-\alpha)(-u_2^{(k)}+u^n)+w_3^{(k-1)}+\rho_w(w_1^{(k)}-w_2^{(k)}))\\
        &\hspace{20pt}-(1-\alpha)(-u_2^{(k+1)}+u^n)+w_3^{(k+1)}
        +\rho_w(w_1^{(k+1)}-w_2^{(k+1)}),w_2^{(k)}-w_2^{(k+1)}\rangle\\
        &=\langle -(1-\alpha)u_2^{(k)}+w_3^{(k)}+(1-\alpha)u_2^{(k+1)}-w_3^{(k+1)}\\
        &\hspace{20pt}-\rho_w(w_1^{(k+1)}-w_2^{(k+1)}),w_2^{(k)}-w_2^{(k+1)}\rangle\\
        &=\langle -(1-\alpha)(e_{u2}^{(k)}-e_{u2}^{(k+1)})
        -\rho_w(e_{w1}^{(k+1)}-e_{w2}^{(k+1)}), e_{w2}^{(k)}-e_{w2}^{(k+1)}\rangle\\
        &=-\rho_w\langle e_{w1}^{(k+1)}-e_{w2}^{(k+1)}, e_{w2}^{(k)}-e_{w2}^{(k+1)}\rangle
        -(1-\alpha)\langle e_{u2}^{(k)}-e_{u2}^{(k+1)},e_{w2}^{(k)}-e_{w2}^{(k+1)}\rangle.
        \end{aligned}
    \end{equation}}%

    Summing up the six inequalities and equations above
    (from \eqref{eq:convex_1} to \eqref{eq:convex_5}), and multiplying both
    sides by $2$, we derive that
    {\small \begin{equation}
        \begin{aligned}
        0&\leq 2\langle -e_{u3}^{(k+1)}
        +\rho_u(e_{u2}^{(k)}-e_{u2}^{(k+1)}),e_{u1}^{(k+1)}\rangle\\
        &\hspace{10pt}+2\langle -e_{w3}^{(k+1)}+\rho_w(e_{w2}^{(k)}-e_{w2}^{(k+1)}),
        e_{w1}^{(k+1)}\rangle+2\langle e_{u3}^{(k+1)},e_{u2}^{(k+1)}\rangle
        +2\langle e_{w3}^{(k+1)}, e_{w2}^{(k+1)}\rangle\\
        &\hspace{10pt}
        -2\rho_u\langle e_{u1}^{(k+1)}-e_{u2}^{(k+1)},e_{u2}^{(k)}-e_{u2}^{(k+1)}\rangle
        -2\rho_w\langle e_{w1}^{(k+1)}-e_{w2}^{(k+1)},e_{w2}^{(k)}-e_{w2}^{(k+1)}\rangle\\
        &=2\rho_u\langle e_{u2}^{(k)}-e_{u2}^{(k+1)}, e_{u1}^{(k+1)}\rangle
        +2\rho_w\langle e_{w2}^{(k)}-e_{w2}^{(k+1)}, e_{w1}^{(k+1)}\rangle\\
        &\hspace{10pt}-2\langle e_{u3}^{(k+1)},e_{u1}^{(k+1)}-e_{u2}^{(k+1)}\rangle
        -2\langle e_{w3}^{(k+1)}, e_{w1}^{(k+1)}-e_{w2}^{(k+1)}\rangle\\
        &\hspace{10pt}
        -2\rho_u\langle e_{u1}^{(k+1)}-e_{u2}^{(k+1)},e_{u2}^{(k)}-e_{u2}^{(k+1)}\rangle
        -2\rho_w\langle e_{w1}^{(k+1)}-e_{w2}^{(k+1)},
        e_{w2}^{(k)}-e_{w2}^{(k+1)}\rangle\\
        &=2\rho_u\langle e_{u2}^{(k)}-e_{u2}^{(k+1)}, e_{u2}^{(k+1)}\rangle
        +2\rho_w\langle e_{w2}^{(k)}-e_{w2}^{(k+1)}, e_{w2}^{(k+1)}\rangle\\
        &\hspace{10pt}
        +\frac{2}{\rho_u}\langle e_{u3}^{(k+1)},e_{u3}^{(k)}-e_{u3}^{(k+1)}\rangle
        +\frac{2}{\rho_w}\langle e_{w3}^{(k+1)}, e_{w3}^{(k)}-e_{w3}^{(k+1)}\rangle\\
        &=\Psi^{(k)}-\Psi^{(k+1)}-\rho_u\|e_{u2}^{(k)}-e_{u2}^{(k+1)}\|_2^2
        -\rho_w\|e_{w2}^{(k)}-e_{w2}^{(k+1)}\|_2^2\\
        &\hspace{10pt}-\frac{1}{\rho_u}\|e_{u3}^{(k)}-e_{u3}^{(k+1)}\|_2^2
        -\frac{1}{\rho_w}\|e_{w3}^{(k)}-e_{w3}^{(k+1)}\|_2^2\\
        &=\Psi^{(k)}-\Psi^{(k+1)}-\rho_u\|e_{u2}^{(k)}-e_{u2}^{(k+1)}\|_2^2\\
        &\hspace{10pt}-\rho_w\|e_{w2}^{(k)}-e_{w2}^{(k+1)}\|_2^2-\rho_u\|e_{u1}^{(k+1)}
        -e_{u2}^{(k+1)}\|_2^2-\rho_w\|e_{w1}^{(k+1)}-e_{w2}^{(k+1)}\|_2^2.
        \end{aligned}
    \end{equation}}%
\end{proof}

The convergence of the iterative sequence can be obtained from the lemmas.
We can prove the following proposition.

\begin{proposition}
\label{prop:convergence}
    With the notations in \eqref{eq:error} and \eqref{eq:Psi}, it holds that
    \begin{equation}
        \|e_{u1}^{(k)}\|_2,\|e_{w1}^{(k)}\|_2,
        \|e_{u2}^{(k)}\|_2,\|e_{w2}^{(k)}\|_2,
        \|e_{u3}^{(k)}\|_2,\|e_{w3}^{(k)}\|_2\to 0,\quad \text{as }k\to\infty.
    \end{equation}
\end{proposition}

\begin{proof}
    By Lemma \ref{lem:Psi}, the sequences
    $$
        \{u_2^{(k)}\},\ \{w_2^{(k)}\},\ \{u_3^{(k)}\},\ \{w_3^{(k)}\},\
        \{u_1^{(k)}-u_2^{(k)}\},\ \{w_1^{(k)}-w_2^{(k)}\}
    $$
    are bounded. Then, $\{u_1^{(k)}\},\{w_1^{(k)}\}$ are bounded. Therefore,
    there exists a subsequence
    $\{(u_1^{(k_\ell)},w_1^{(k_\ell)},u_2^{(k_\ell)},w_2^{(k_\ell)}$,
    $u_3^{(k_\ell)},w_3^{(k_\ell)})\}$ converging, and let the limit be
    $$
        (u_1^{(\infty)},w_1^{(\infty)},u_2^{(\infty)},w_2^{(\infty)},
        u_3^{(\infty)},w_3^{(\infty)}).
    $$
    Taking the limit along the subsequence in the equations
    \eqref{eq:linear}, \eqref{eq:nonlinear}, \eqref{eq:update}, we find that
    $$
        (u_1^{(\infty)},w_1^{(\infty)},u_2^{(\infty)},w_2^{(\infty)},
        u_3^{(\infty)},w_3^{(\infty)})
    $$
    is exactly the unique stationary point of the function $\mathcal{L}$, i.e.,
    \begin{equation}
        (u_1^{(\infty)},w_1^{(\infty)},u_2^{(\infty)},w_2^{(\infty)},
        u_3^{(\infty)},w_3^{(\infty)})
        =(u_1^*,w_1^*,u_2^*,w_2^*,u_3^*,w_3^*).
    \end{equation}
    Since the sequence $\{\Psi^{(k)}\}$ is monotonically decreasing and has a
    subsequence $\{\Psi^{(k_\ell)}\}$ converging to $0$, the sequence
    $\{\Psi^{(k)}\}$ converges to $0$ itself. Then, it holds that, as
    $k\to\infty$,
    \begin{equation}
        \|e_{u1}^{(k)}-e_{u2}^{(k)}\|_2,\|e_{w1}^{(k)}-e_{w2}^{(k)}\|_2,
        \|e_{u2}^{(k)}\|_2,\|e_{w2}^{(k)}\|_2,
        \|e_{u3}^{(k)}\|_2,\|e_{w3}^{(k)}\|_2\to 0.
    \end{equation}
    Utilizing
    \begin{equation}
         \|e_{u1}^{(k)}\|_2\leq
         \|e_{u1}^{(k)}-e_{u2}^{(k)}\|_2+\|e_{u2}^{(k)}\|_2,\quad
         \|e_{w1}^{(k)}\|_2\leq
         \|e_{w1}^{(k)}-e_{w2}^{(k)}\|_2+\|e_{w2}^{(k)}\|_2,
    \end{equation}
    concludes the proof.
\end{proof}

From the algorithm, we can see that the bound-preserving property is satisfied.
It holds that $-1<u_2^{(k)}<1$ at a point-wise level for all $k\in\mathbb{N}$
because of the logarithmic term.
We only need $u_2^{(k)}$ to be a correct approach to the solution $u^{n+1}$. The
following theorem follows directly from the proposition.

\begin{theorem}
    Let $\{u_2^{(k)}\}$ be the iterative sequence generated by Algorithm
    \ref{alg:first-order}, it holds that
    \begin{equation}
        \lim_{k\to\infty}\|u_2^{(k)}-u^{n+1}\|_2=0,
    \end{equation}
    where $u^{n+1}$ is the solution of the first-order convex splitting scheme
    at $t_{n+1}$.
\end{theorem}

Below, we examine how the mass of the numerical solution changes and demonstrate that this algorithm nearly preserves mass conservation.

\begin{proposition}\label{Pro1}
    With the notations in \eqref{eq:error},\eqref{eq:Psi} and \eqref{eq:mass},
    it holds that
    \begin{equation}
        |\mathcal{M}(u_2^{(k+1)})-\mathcal{M}(u^n)|
        \leq\alpha\|u_1^{(k+1)}-u_2^{(k+1)}\|_2+\rho_w\|w_2^{(k)}-w_2^{(k+1)}\|_2,
    \end{equation}
    where the average mass $\mathcal{M}$ of $u \in\cper$ is defined as
    \begin{equation}
        \label{eq:mass}
        \mathcal{M}(u)=\langle u,1\rangle.
    \end{equation}
\end{proposition}

\begin{proof}
    Summing up the two equations \eqref{eq:linear_2} and \eqref{eq:nonlinear_2},
    we have
    \begin{equation}
        u_2^{(k+1)}-u^n=\tau\Delta_h w_1^{(k+1)}-\alpha(u_1^{(k+1)}-u_2^{(k+1)})
        +\rho_w(w_2^{(k)}-w_2^{(k+1)}).
    \end{equation}
    Since $\mathcal{M}$ is a linear operator, we can get that
    \begin{equation}
        \begin{aligned}
        &|\mathcal{M}(u_2^{(k+1)})-\mathcal{M}(u^n))|
        =|\mathcal{M}(u_2^{(k+1)}-u^n)|\\
        &=|\mathcal{M}(\tau\Delta_h w_1^{(k+1)}-\alpha(u_1^{(k+1)}-u_2^{(k+1)})
        +\rho_w(w_2^{(k)}-w_2^{(k+1)})|\\
        &=|\tau\mathcal{M}(\Delta_h w_1^{(k+1)})
        -\alpha\mathcal{M}(u_1^{(k+1)}-u_2^{(k+1)})
        +\rho_w\mathcal{M}(w_2^{(k)}-w_2^{(k+1)})|\\
        &=|-\alpha\mathcal{M}(u_1^{(k+1)}-u_2^{(k+1)})
        +\rho_w\mathcal{M}(w_2^{(k)}-w_2^{(k+1)})|\\
        &\leq \alpha\|u_1^{(k+1)}-u_2^{(k+1)}\|_1+\rho_w\|w_2^{(k)}-w_2^{(k+1)}\|_1\\
        &\leq \alpha\|u_1^{(k+1)}-u_2^{(k+1)}\|_2+\rho_w\|w_2^{(k)}-w_2^{(k+1)}\|_2.
        \end{aligned}
    \end{equation}
\end{proof}

\begin{remark}
    In the framework of ADMM, a reasonable convergence criterion is
    {\small \begin{equation}
    \label{eq:criterion}
        \|u_1^{(k+1)}-u_2^{(k+1)}\|_2+\|w_1^{(k+1)}-w_2^{(k+1)}\|_2
        +\|u_1^{(k)}-u_1^{(k+1)}\|_2+\|w_1^{(k)}-w_1^{(k+1)}\|_2\leq \gamma,
    \end{equation}}%
    where $\gamma>0$ is the stopping tolerance. Proposition \ref{Pro1} indicates
    that the magnitude of mass variation is on the same order as the stopping
    tolerance.
\end{remark}

%\begin{remark}
%    This result shows that the variation in mass is of the same order as the
%    stopping tolerance. Therefore, the algorithm is almost energy-conserving
%    when the stopping tolerance is of the same order as machine precision.
%\end{remark}

The solver for the second-order scheme has similar convergence and mass
conservation properties.

\begin{theorem}
    Let $\{u_2^{(k)}\}$ be the iterative sequence generated by Algorithm
    \ref{alg:second-order}, it holds that
    \begin{equation}
        \lim_{k\to\infty}\|u_2^{(k)}-u^{n+1}\|_2=0,
    \end{equation}
    where $u^{n+1}$ is the solution of the second-order convex splitting scheme
    at the next time step.
\end{theorem}

\begin{proposition}
    Let $\{(u_1^{(k)},w_1^{(k)},u_2^{(k)},w_2^{(k)})\}$ be the sequence
    generated by Algorithm \ref{alg:second-order}.
    With the notations in \eqref{eq:mass},
    it holds that
    \begin{equation}
        |\mathcal{M}(u_2^{(k+1)})-\mathcal{M}(u^n)|
        \leq\alpha\|u_1^{(k+1)}-u_2^{(k+1)}\|_2+\rho_w\|w_2^{(k)}-w_2^{(k+1)}\|_2.
    \end{equation}
\end{proposition}

In order to keep the variation in mass as small as possible, a sufficiently
small stopping tolerance is required. After satisfying the iteration
termination criterion \eqref{eq:criterion}, we finally take $u_2^{(k)}$ of the
current iteration step as the numerical approach to $u^{n+1}$ to complete the
algorithm. 
\section{Numerical examples}

\label{sec:numerical}

In this section, we provide several numerical examples. Section
\ref{sec:numerical_1} presents a convergence test of the convex splitting
discrete scheme to verify the accuracy of our solver. Numerical simulations of
coarsening processes in 2D and 3D are demonstrated in Section
\ref{sec:numerical_2} and Section \ref{sec:numerical_3}, respectively,
illustrating the robustness of our solver.
%We
%primarily aim to validate the correctness of our theory so that the actual
%calculation time is not included in the results.

\subsection{Convergence test for the discrete scheme}

\label{sec:numerical_1}

In this numerical examples, we use a slightly different formulation of
the logarithmic energy potential
{\small \begin{equation}
    \label{eq:potential}
    E(u)=\int_\Omega\left(\frac{1}{2\theta_0}((1+u)\log(1+u)+(1-u)\log(1-u)
    -\frac{1}{2}(u^2-1))+\frac{\epsilon^2}{2}|\nabla u|\right)
\end{equation}}%
as in the numerical examples of \cite{chen2019positivity}. This change is only
technical in order to make comparisons with numerical results in the literature.
The energy potential \eqref{eq:potential} is equivalent to the
\eqref{eq:flory-huggins} with some modified parameters.

The setting of this example is the same as the asymptotic convergence test in
\cite{chen2019positivity}. The initial condition is
\begin{equation}
    u_0(x,y)=1.8\left(\frac{1-\cos\left(\frac{4x\pi}{3.2}\right)}{2}\right)
    \left(\frac{1-\cos\left(\frac{4y\pi}{3.2}\right)}{2}\right)-0.9.
\end{equation}
The ``Cauchy difference'' $\delta_u$ is computed between approximate solutions
obtained with successively finer mesh sizes to compute the convergence rate.
The parameters are
\begin{equation}
    L=3.2,\ \epsilon=0.2,\ \theta_0=3,\ T=0.4.
\end{equation}
The refinement path for the first-order scheme is quadratic,
$\tau=0.4h^2$, and the refinement path for the second-order scheme is linear,
$\tau=0.8h$. The stabilized parameter of the second-order scheme $A$ is set as
$1/16$. The test results, displayed in Table \ref{tab:res}, confirm the
predicted accuracy: first order in time and second order in space for the
first-order scheme, and second order in both time and space for the second-order
scheme.

The coefficients in the solver are
\begin{equation}
    \rho_u=\tau^{-1/2},\quad \rho_w=\tau^{1/2},\quad \alpha=1/2.
\end{equation}
The convergence criterion is set as
\begin{equation}
    \sqrt{h^2\|u_1^{(k)}-u_2^{(k)}\|_2^2+h^2\|w_1^{(k)}-w_2^{(k)}\|_2^2}\leq
    \gamma=10^{-10}.
\end{equation}

The number of iterations, the change in total mass, the energy, $1-\max(u)$ and
$1+\min(u)$ over time are given in Figure \ref{fig:other1}, Figure
\ref{fig:other2}, Figure \ref{fig:other3} and Figure \ref{fig:other4}, which
confirm the energy stable and the bound preserving properties.

\begin{table}[!t]
\centering
\begin{tabular}{|c|c|cc|ccc|}
\hline
\multicolumn{2}{|c|}{} & \multicolumn{2}{c|}{First order scheme} &
\multicolumn{2}{c|}{Second order scheme}\\\hline
$h_c$      & $h_f$     & \multicolumn{1}{c|}{$\|\delta_u\|_2$} & Rate  &
\multicolumn{1}{c|}{$\|\delta_u\|_2$} & \multicolumn{1}{c|}{Rate} \\ \hline
$3.2/16$          & $3.2/32$  & \multicolumn{1}{c|}{5.38E$-$1}      & -     &
\multicolumn{1}{c|}{4.89E$-$2}      & \multicolumn{1}{c|}{-}        \\ \hline
$3.2/32$          & $3.2/64$  & \multicolumn{1}{c|}{1.55E$-$2}      & 1.794 &
\multicolumn{1}{c|}{3.16E$-$2}      & \multicolumn{1}{c|}{0.630}  \\ \hline
$3.2/64$          & $3.2/128$ & \multicolumn{1}{c|}{4.02E$-$3}      & 1.948 &
\multicolumn{1}{c|}{6.33E$-$3}      & \multicolumn{1}{c|}{2.320}  \\ \hline
$3.2/128$         & $3.2/256$ & \multicolumn{1}{c|}{1.01E$-$3}      & 1.987 &
\multicolumn{1}{c|}{1.41E$-$3}      & \multicolumn{1}{c|}{2.163}  \\ \hline
$3.2/256$         & $3.2/512$ & \multicolumn{1}{c|}{2.54E$-$4}  & 1.997     &
\multicolumn{1}{c|}{3.35E$-$4}      & \multicolumn{1}{c|}{2.079}  \\ \hline
\end{tabular}
\caption{Errors and convergence rates. $h_c$ is the coarse mesh size, $h_f$ is
the fine mesh size.}
    \label{tab:res}
\end{table}

\begin{figure}[!t]
    \centering
    \includegraphics[width=0.5\linewidth]{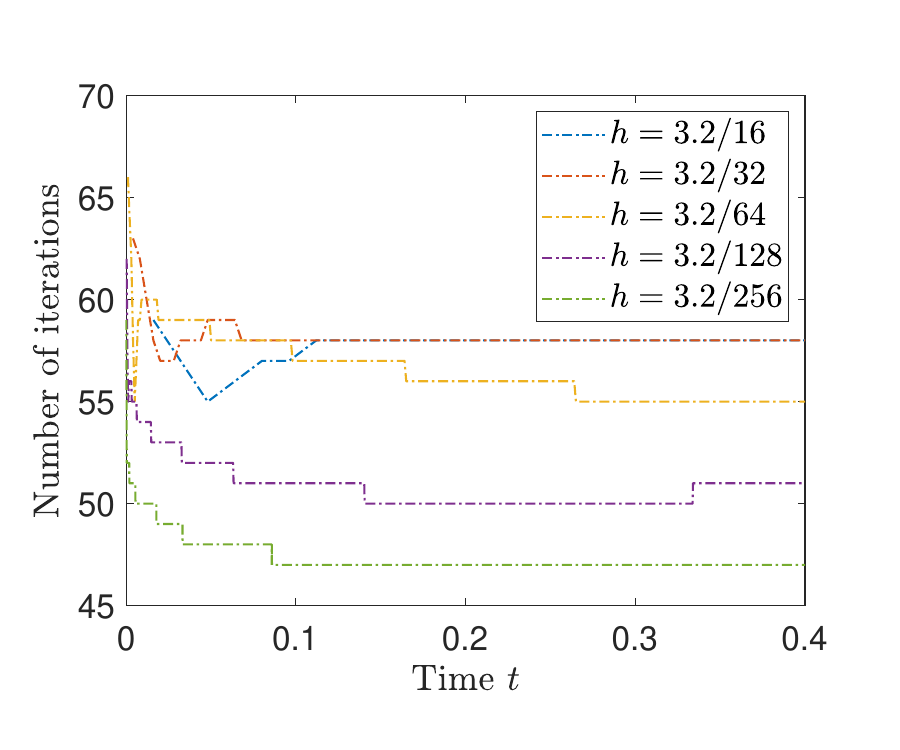}%
    \includegraphics[width=0.5\linewidth]{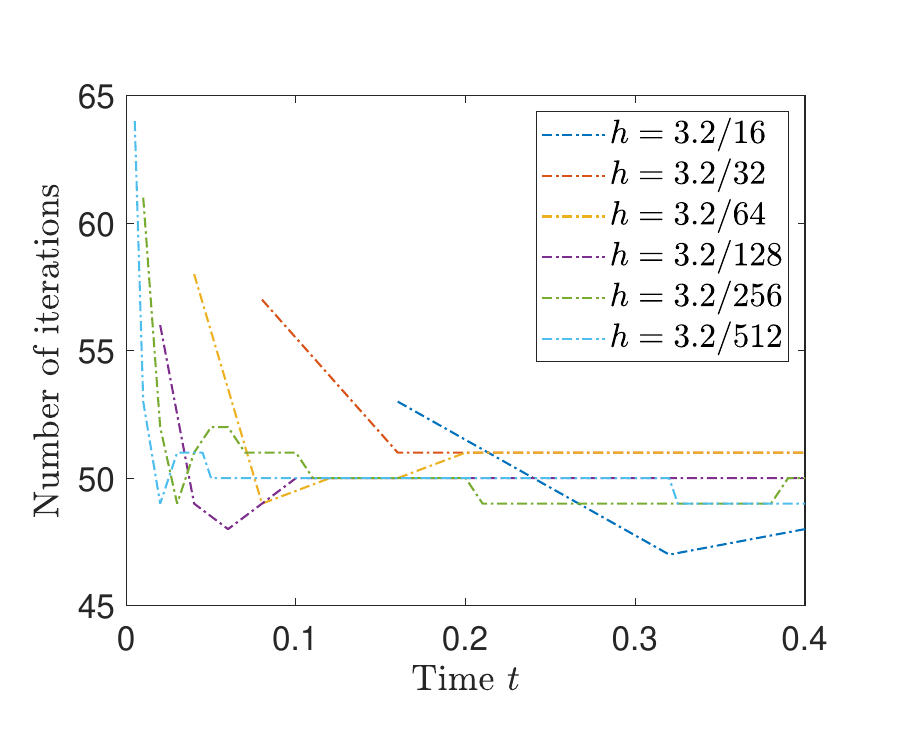}
\caption{(Convergence test)
The number of iterations over time $t$. Left: first-order scheme,
right: second-order scheme.}
\label{fig:other1}
\end{figure}

\begin{figure}[!t]
    \centering
    \includegraphics[width=0.5\linewidth]{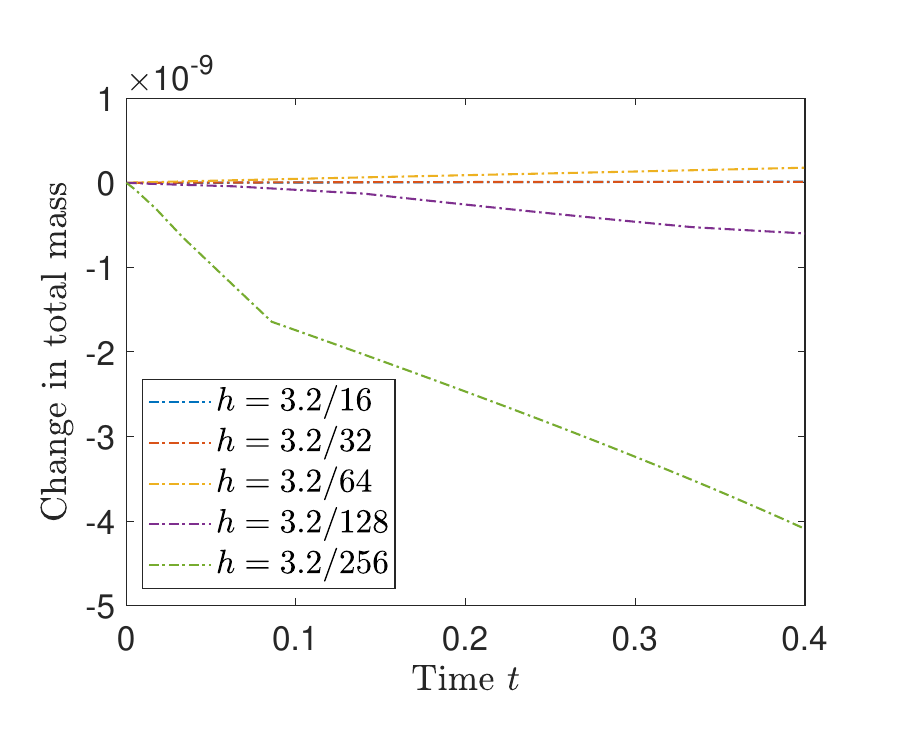}%
    \includegraphics[width=0.5\linewidth]{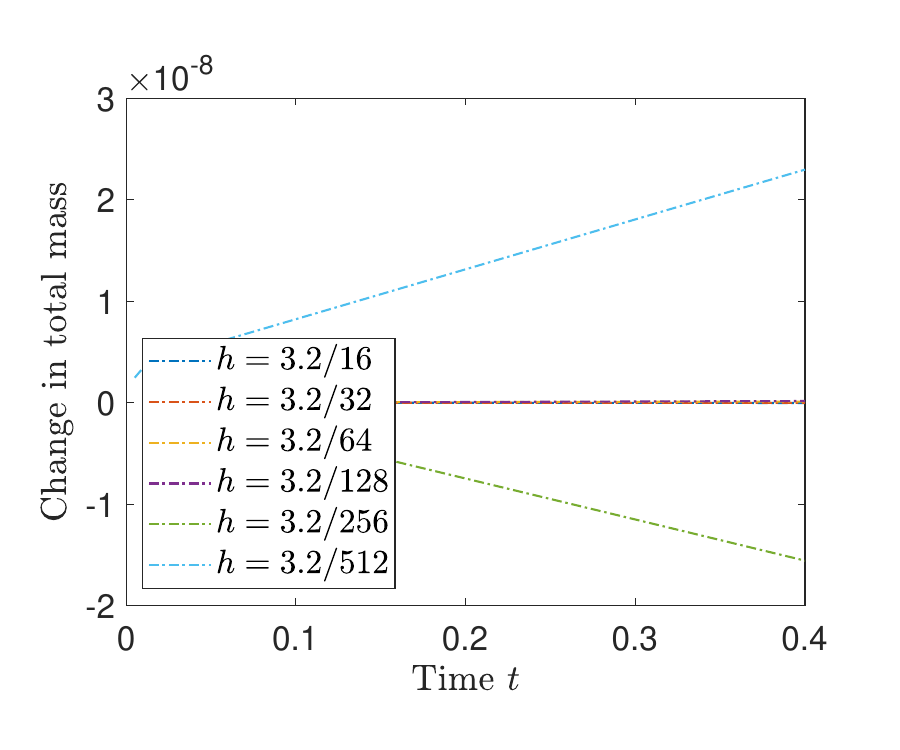}
\caption{(Convergence test) The change in total mass over time $t$. Left:
first-order scheme, right: second-order
scheme.}
\label{fig:other2}
\end{figure}

\begin{figure}[!t]
    \centering
    \includegraphics[width=0.5\linewidth]{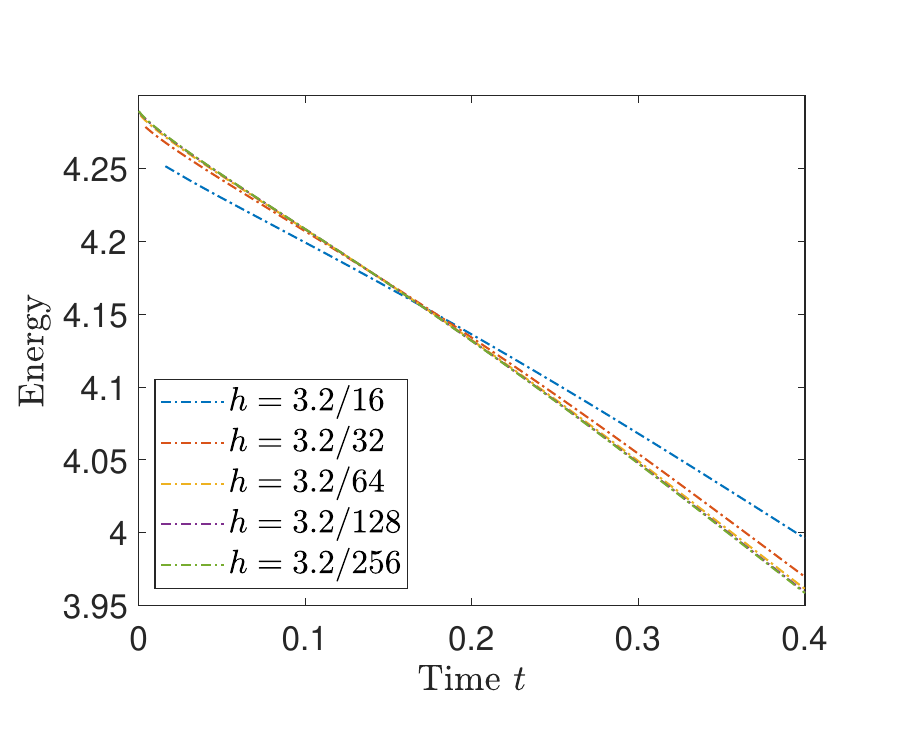}%
    \includegraphics[width=0.5\linewidth]{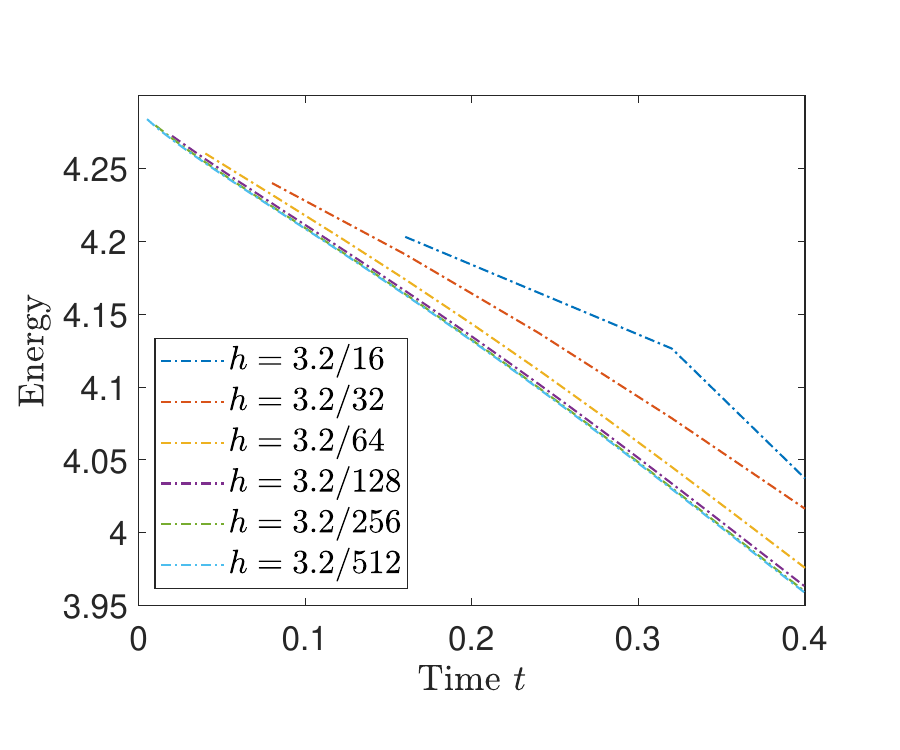}
\caption{(Convergence test) The energy over time $t$. Left: first-order scheme,
right: second-order scheme.}
\label{fig:other3}
\end{figure}

\begin{figure}[!t]
    \centering
    \includegraphics[width=0.5\linewidth]{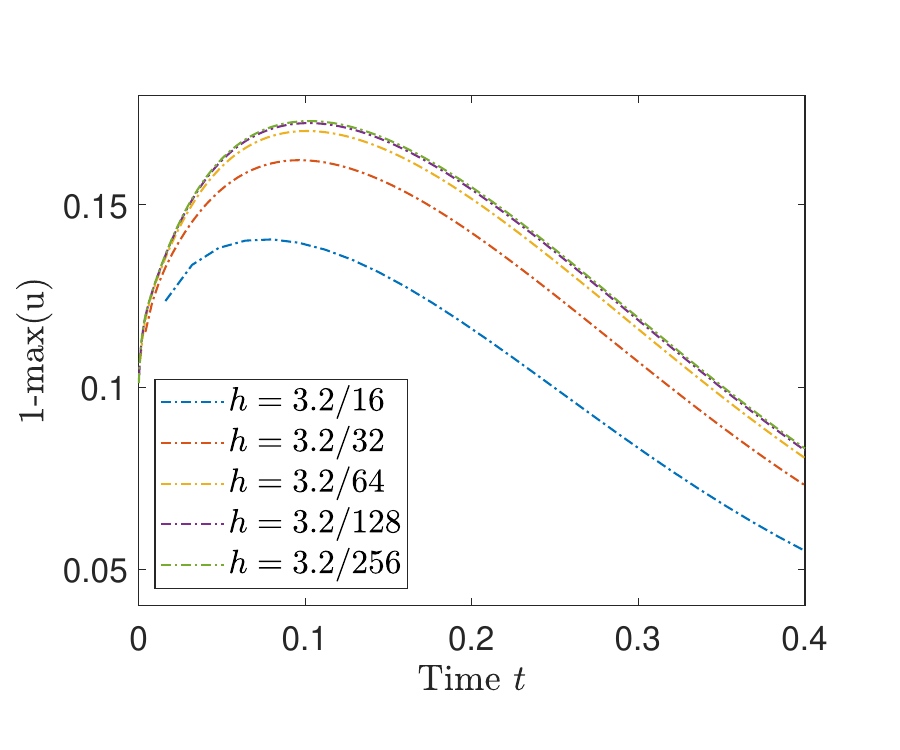}%
    \includegraphics[width=0.5\linewidth]{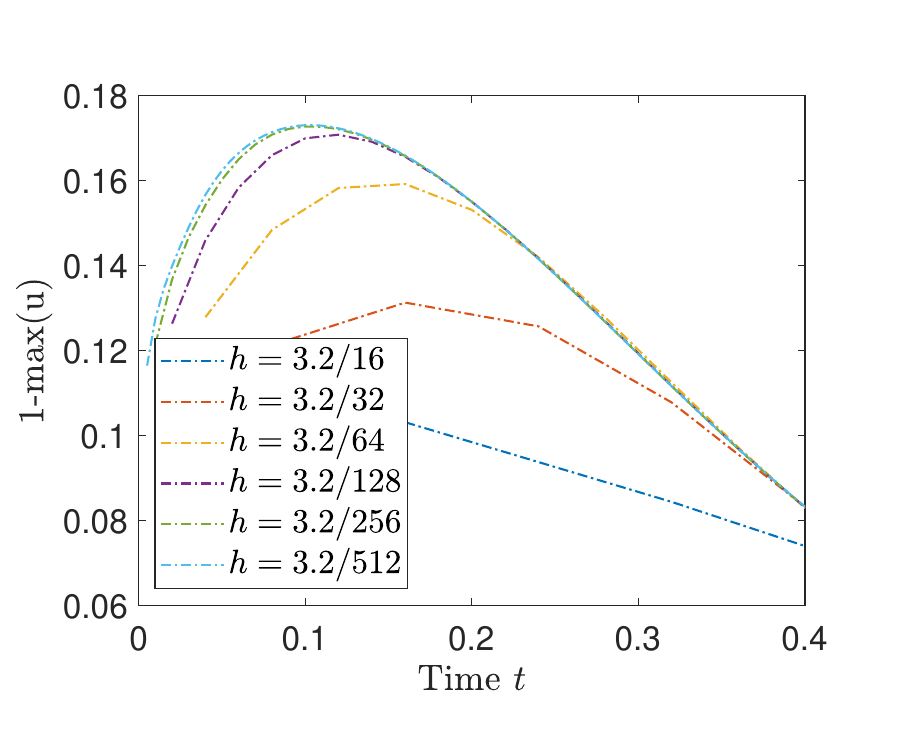}\\
    \includegraphics[width=0.5\linewidth]{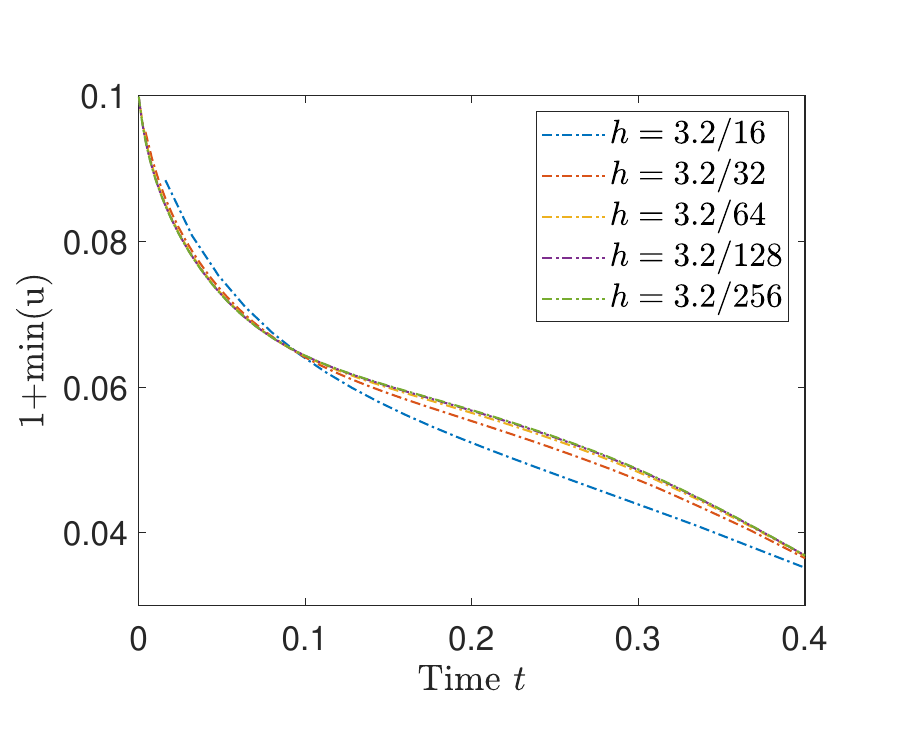}%
    \includegraphics[width=0.5\linewidth]{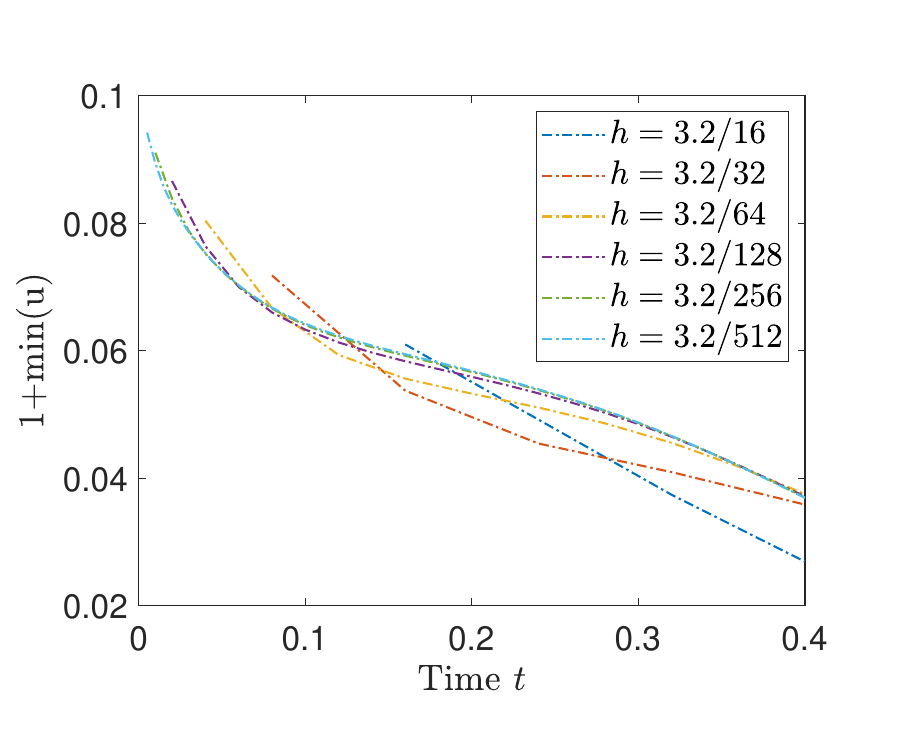}
\caption{(Convergence test) $1-\max(u)$ and $1+\min(u)$ over time $t$. Left:
first-order scheme, right: second-order
scheme.}
\label{fig:other4}
\end{figure}

\subsection{Two-dimensional simulation of coarsening process}
\label{sec:numerical_2}

We show a two-dimensional simulation of coarsening processes here. The initial
condition is
\begin{equation}
    u_{i,j}^0=0.15+r_{i,j},\quad i,j\in\{1,2,\cdots,N\},
\end{equation}
where $r_{i,j}$ is a uniformly distributed random variable from the interval
$[0,0.1]$, and the parameters are
\begin{equation}
    L=1,\ \epsilon=0.01,\ \theta_0=3,\ T=1,\ \tau=0.001,\ N=128.
\end{equation}
This example demonstrates the numerical results of the coarsening process
starting from smooth initial conditions. The coefficients in the solver are
\begin{equation}
    \rho_u=1,\quad \rho_w=1,\quad \alpha=1/2.
\end{equation}
The convergence criterion is set as
\begin{equation}
    \sqrt{h^2\|u_1^{(k)}-u_2^{(k)}\|_2^2+h^2\|w_1^{(k)}-w_2^{(k)}\|_2^2}\leq
    \gamma=10^{-8}.
\end{equation}
We use the second-order scheme here. From the numerical results in Figure
\ref{fig:coarsening2}, we can see that the phase field $u$ evolves from the
initial condition to a coarsened state. The phase field $u$ is coarsened at
different time steps, and the coarsening process is consistent with the physical
phenomenon of phase separation. The numerical result demonstrates the robustness
of our solver in simulating the coarsening.

The number of iterations, the change in total mass, the energy, $1-\max(u)$ and
$1+\min(u)$ over time are given in Figure \ref{fig:other21}, and Figure
\ref{fig:other24}, which
confirm the energy stable and the bound preserving properties.

\begin{figure}[!t]
    \includegraphics[width=0.25\linewidth]{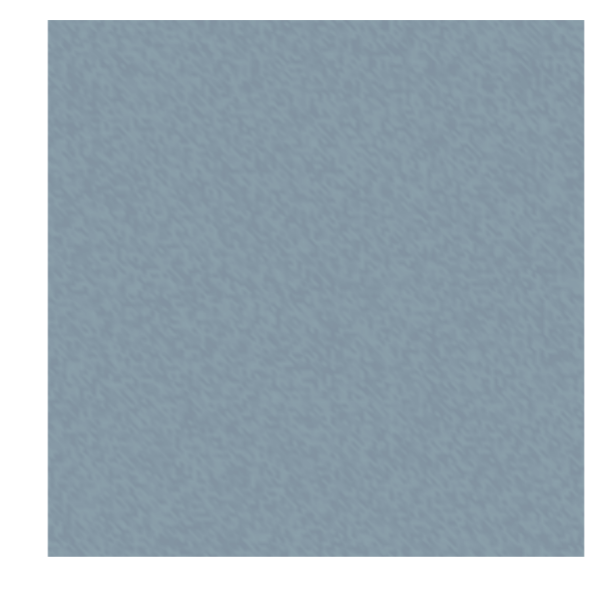}%
    \includegraphics[width=0.25\linewidth]{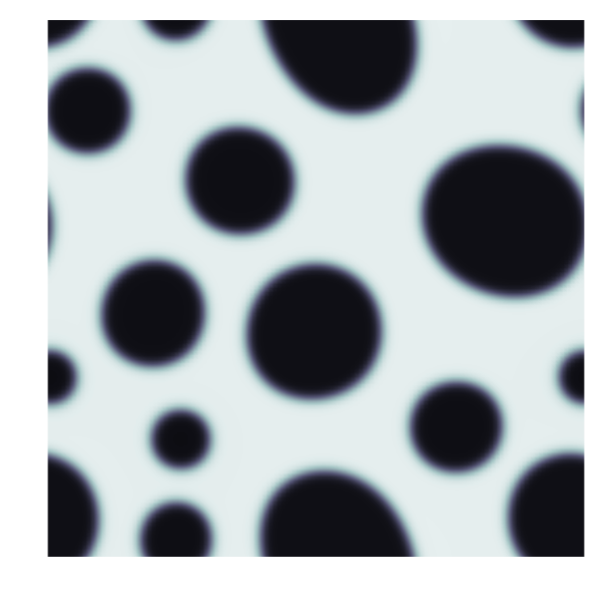}%
    \includegraphics[width=0.25\linewidth]{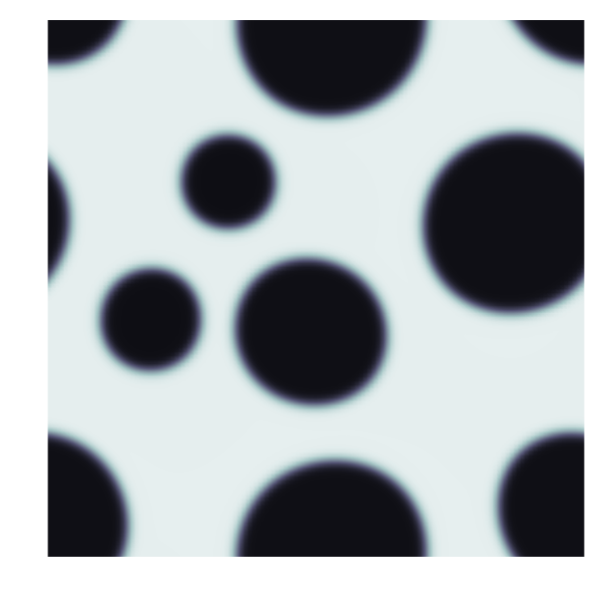}\\
    \includegraphics[width=0.25\linewidth]{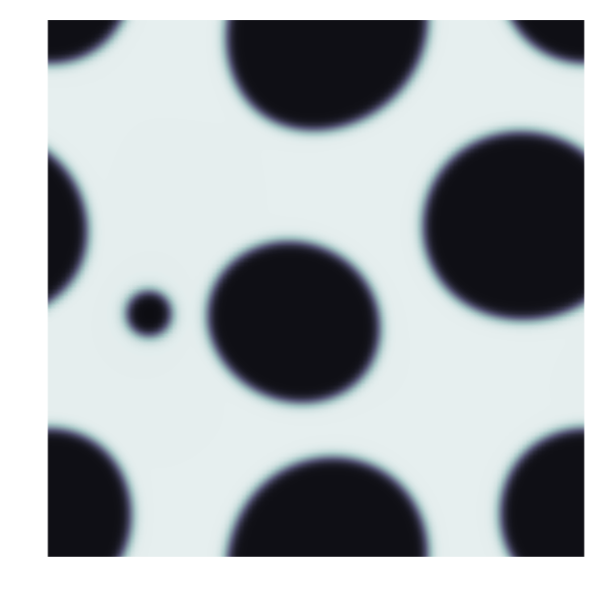}%
    \includegraphics[width=0.25\linewidth]{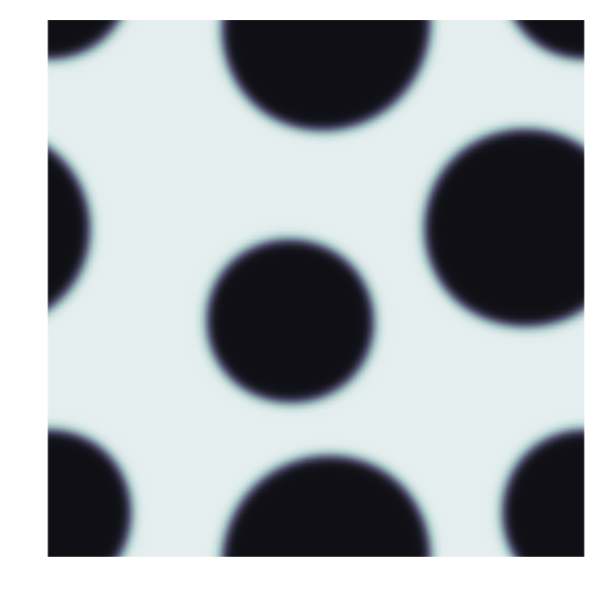}%
    \includegraphics[width=0.25\linewidth]{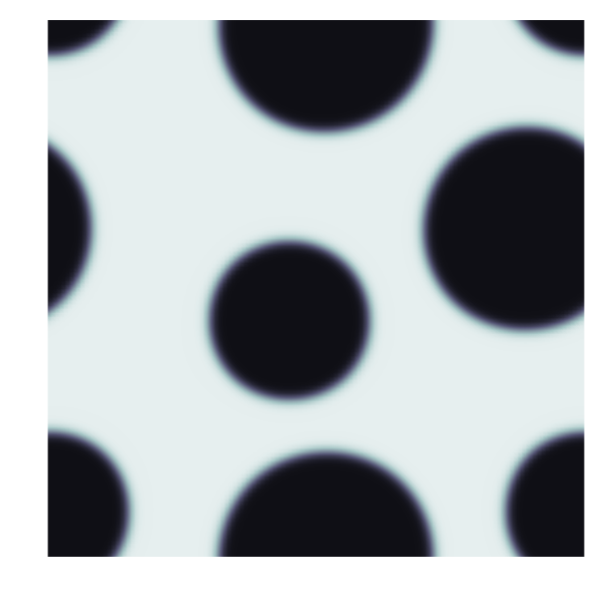}
\caption{(2D coarsening process) The coarsening process of the phase field $u$
at different time steps. From top to bottom, left to right,
$t=0,0.2,0.4,0.6,0.8,1$. The numerical solution is close to $1$ in the white
region and close to $-1$ in the
black region.}
\label{fig:coarsening2}
\end{figure}

\begin{figure}[!t]
    \centering
    \includegraphics[width=0.33\linewidth]{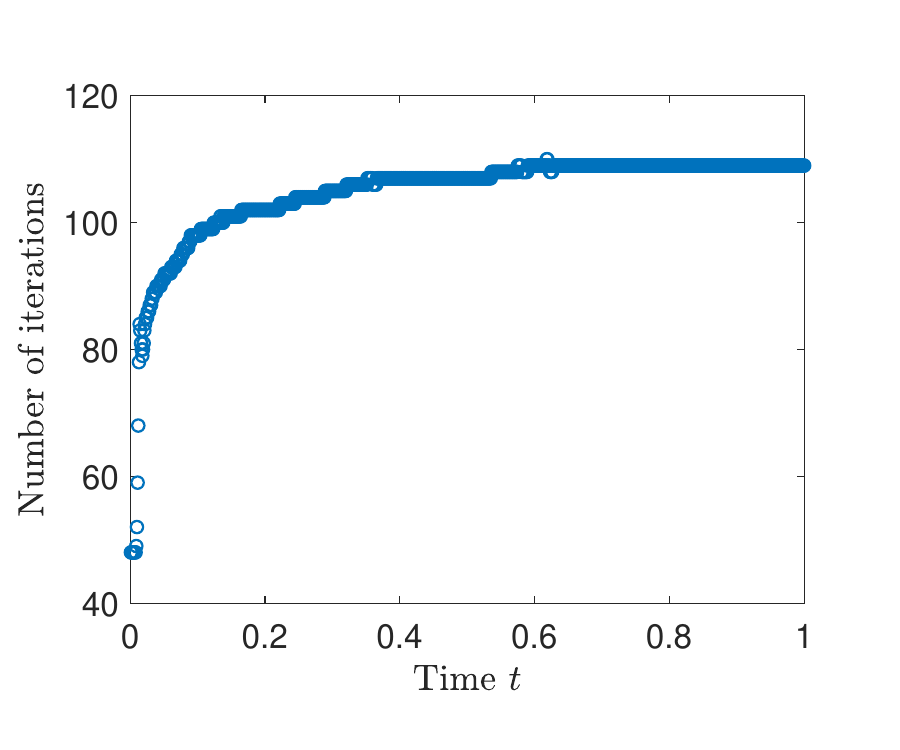}%
    \includegraphics[width=0.33\linewidth]{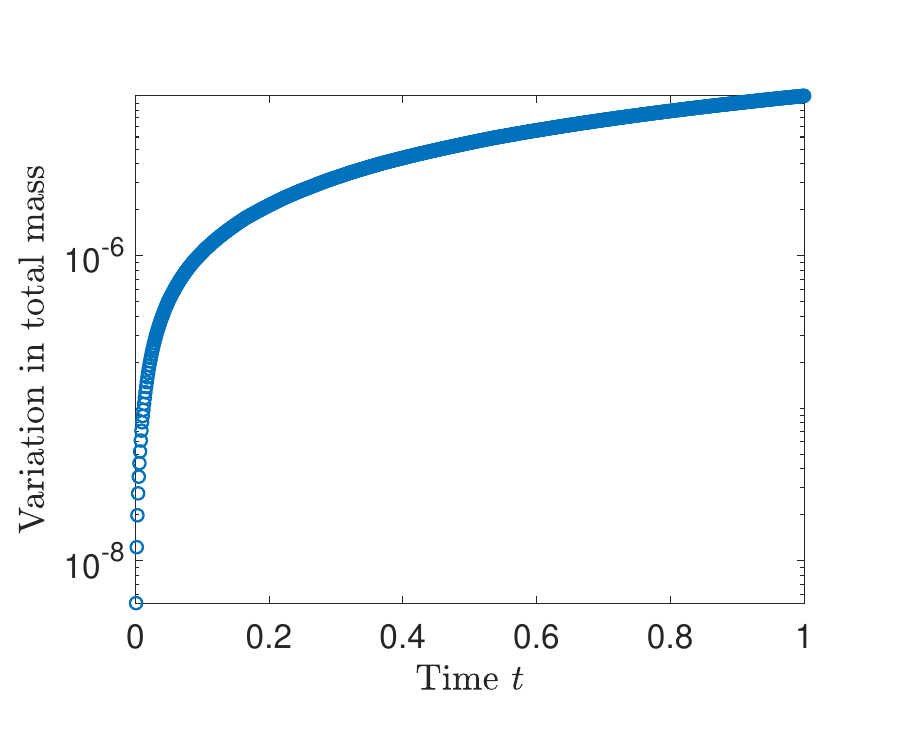}%
    \includegraphics[width=0.33\linewidth]{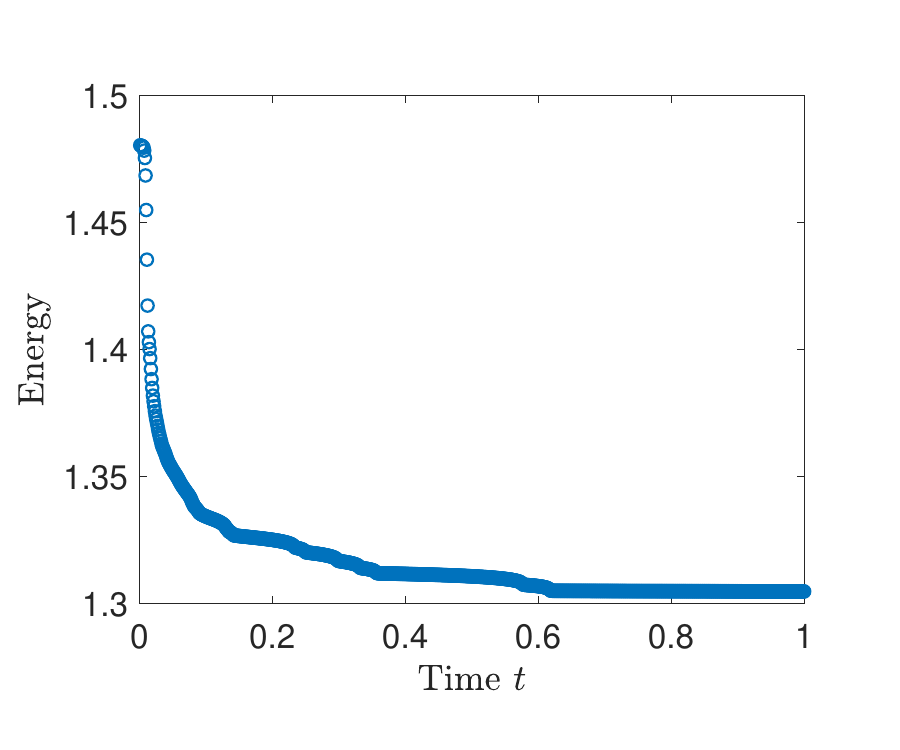}
\caption{(2D coarsening process)
Left: the number of iterations over time $t$, middle: the variation in total
mass over time $t$, right: the energy over time $t$.}
\label{fig:other21}
\end{figure}

\begin{figure}[!t]
    \centering
    \includegraphics[width=0.4\linewidth]{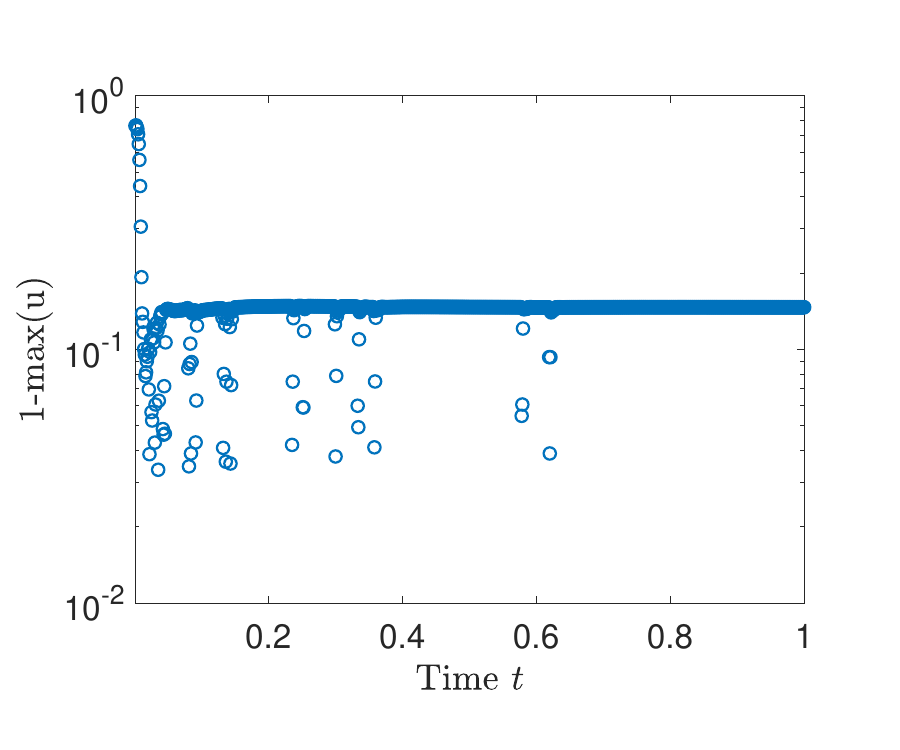}%
    \includegraphics[width=0.4\linewidth]{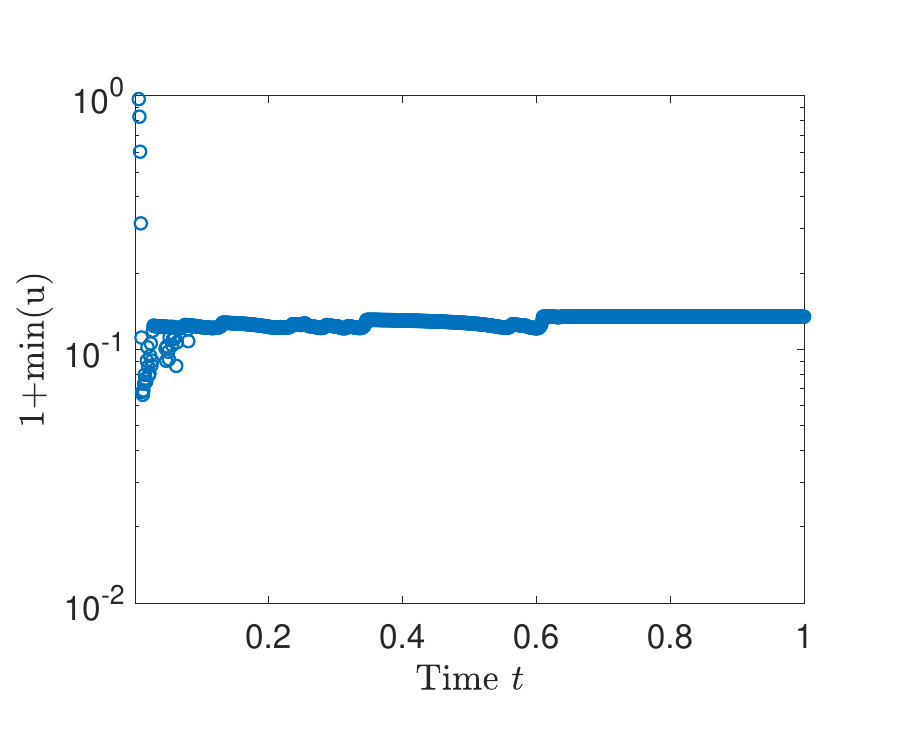}
\caption{(2D coarsening process)
$1-\max(u)$ and $1+\min(u)$ over time $t$.}
\label{fig:other24}
\end{figure}

\subsection{Three-dimensional simulation of coarsening process}
\label{sec:numerical_3}

We show a three-dimensional simulation of coarsening processes here. The initial
condition is
\begin{equation}
    u_{i,j,\ell}^0=r_{i,j,\ell},\quad i,j,\ell\in\{1,2,\cdots,N\},
\end{equation}
where $r_{i,j,\ell}$ is a uniformly distributed random variable from the
interval $[-0.05,0.05]$, and the parameters are
\begin{equation}
    L=1,\ \epsilon=0.01,\ \theta_0=3,\ T=2,\ \tau=0.01,\ N=64.
\end{equation}
The coefficients in the solver are
\begin{equation}
    \rho_u=\rho_w=1,\quad \alpha=1/2.
\end{equation}
The convergence criterion is set as
\begin{equation}
    \sqrt{h^2\|u_1^{(k)}-u_2^{(k)}\|_2^2+h^2\|w_1^{(k)}-w_2^{(k)}\|_2^2}\leq
    \gamma=10^{-8}.
\end{equation}
We use the second-order scheme here. Notice that the time step $\tau=0.01$ is
significantly large, which verifies that our method has no requirement for the
time step. From the numerical results in Figure
\ref{fig:coarsening}, we can see that the phase field $u$ evolves from the
initial condition to a coarsened state. The phase field $u$ is coarsened at
different time steps, and the coarsening process is consistent with the physical
phenomenon of phase separation. The numerical result demonstrates the robustness
of our solver in simulating the coarsening. Also, the energy stable property is
confirmed by the energy over time. We can see that the energy is decreasing over
time, as shown in Figure \ref{fig:coarsening1}. It is observed that the rate of
energy decay is exponential, which aligns with the physical reality.

The number of iterations, the change in total mass, $1-\max(u)$ and
$1+\min(u)$ over time are given in Figure \ref{fig:other31} and Figure
\ref{fig:other34}, which
confirm the energy stable and the bound preserving properties.

\begin{figure}[!t]
    \includegraphics[width=0.3\linewidth]{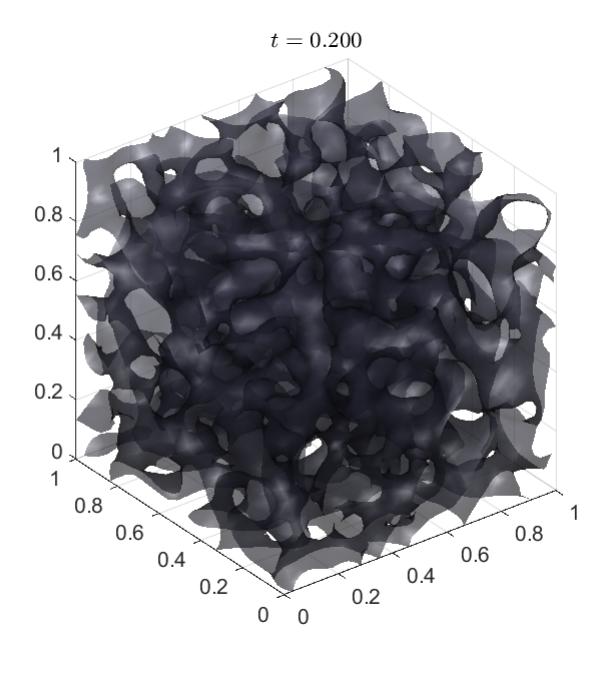}%
    \includegraphics[width=0.3\linewidth]{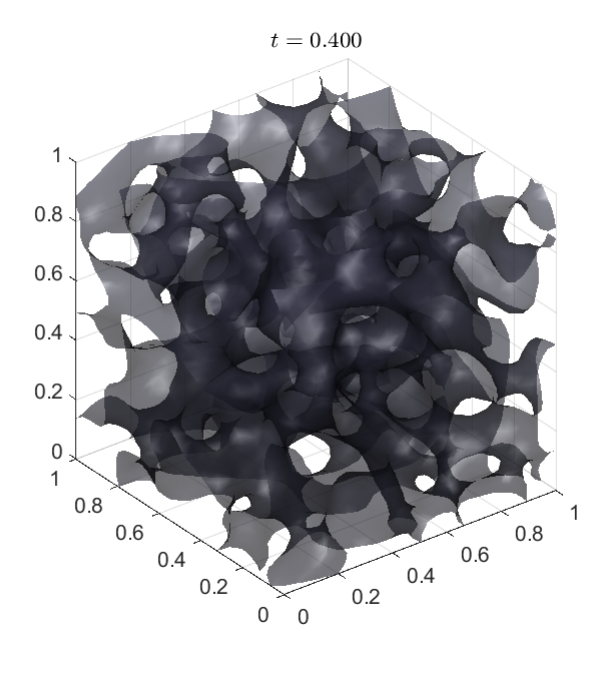}%
    \includegraphics[width=0.3\linewidth]{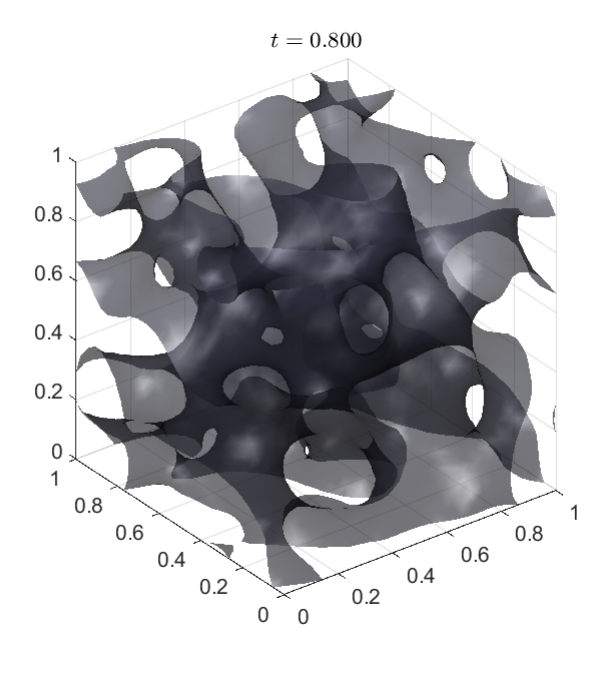}\\
    \includegraphics[width=0.3\linewidth]{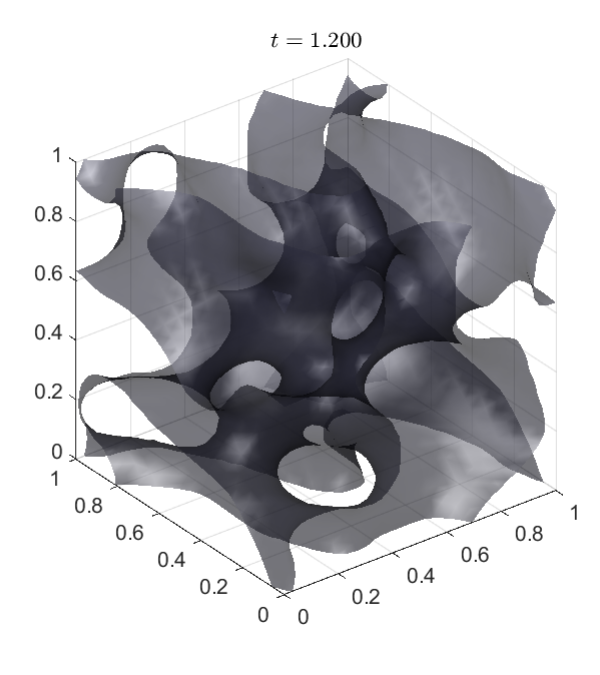}%
    \includegraphics[width=0.3\linewidth]{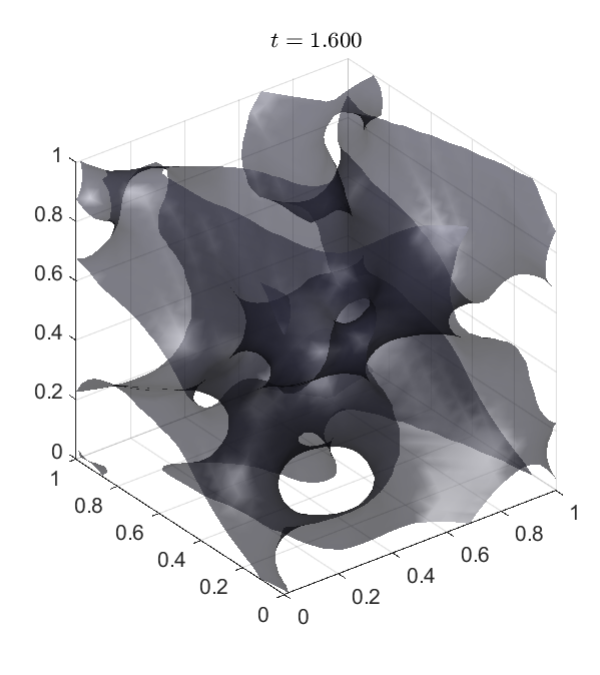}%
    \includegraphics[width=0.3\linewidth]{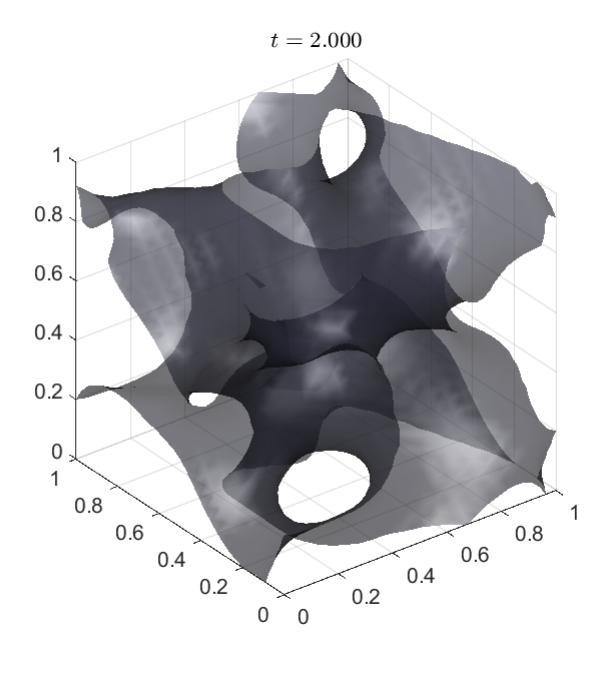}
\caption{(3D coarsening process) The coarsening process of the phase field $u$
at different time steps. The figure shows the interface of $u=0$ by
interpolating the numerical solution.}
\label{fig:coarsening}
\end{figure}

\begin{figure}[!t]
    \centering
    \includegraphics[width=0.4\linewidth]{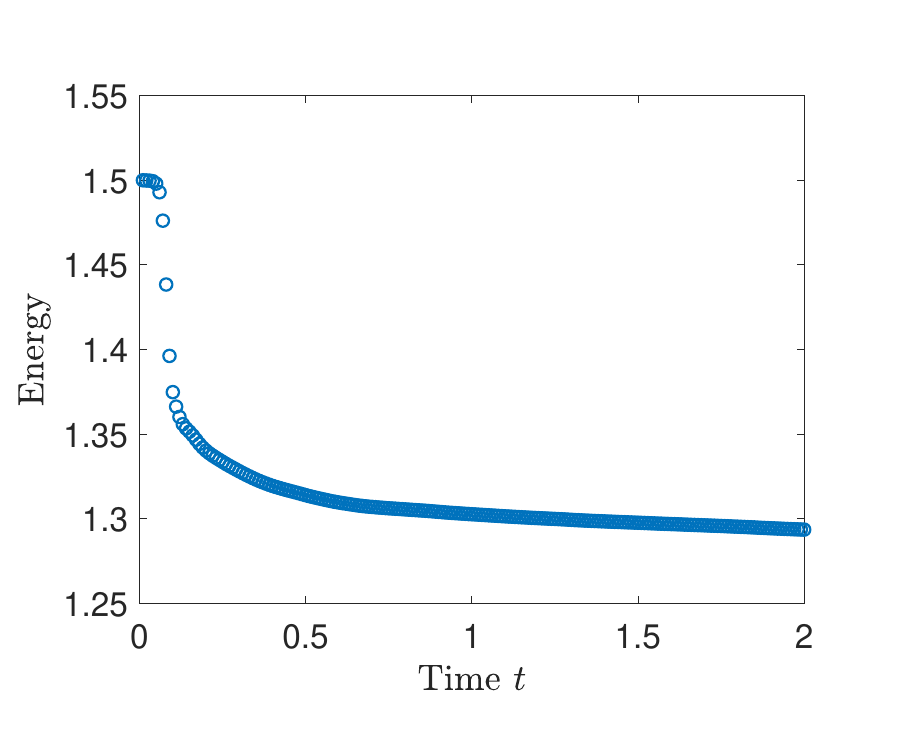}%
    \includegraphics[width=0.4\linewidth]{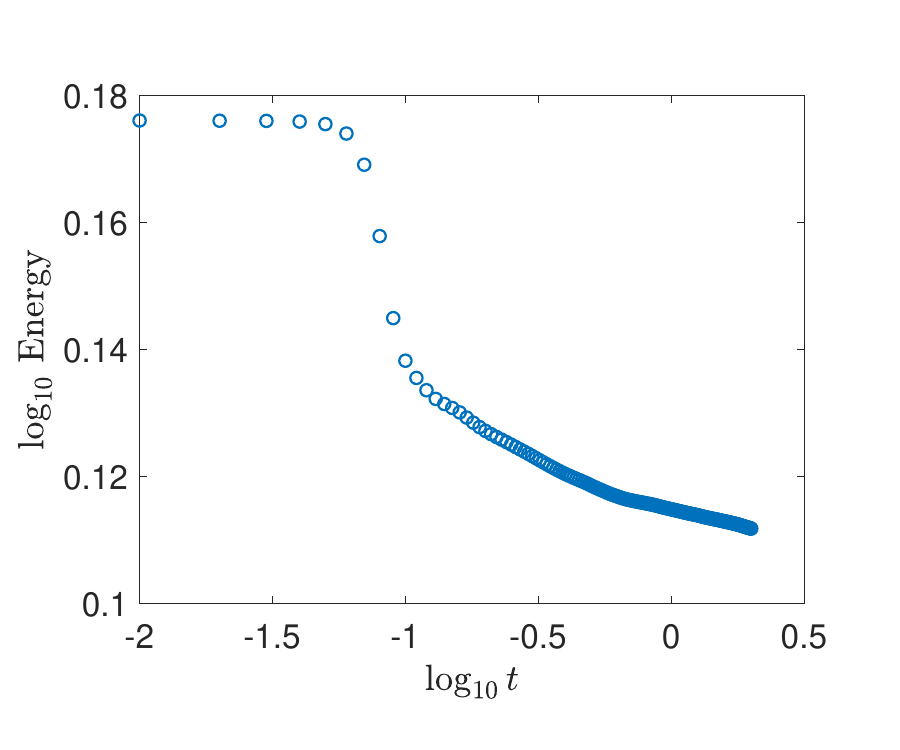}
\caption{(3D coarsening process) The energy over time $t$.}
\label{fig:coarsening1}
\end{figure}

\begin{figure}[!t]
    \centering
    \includegraphics[width=0.4\linewidth]{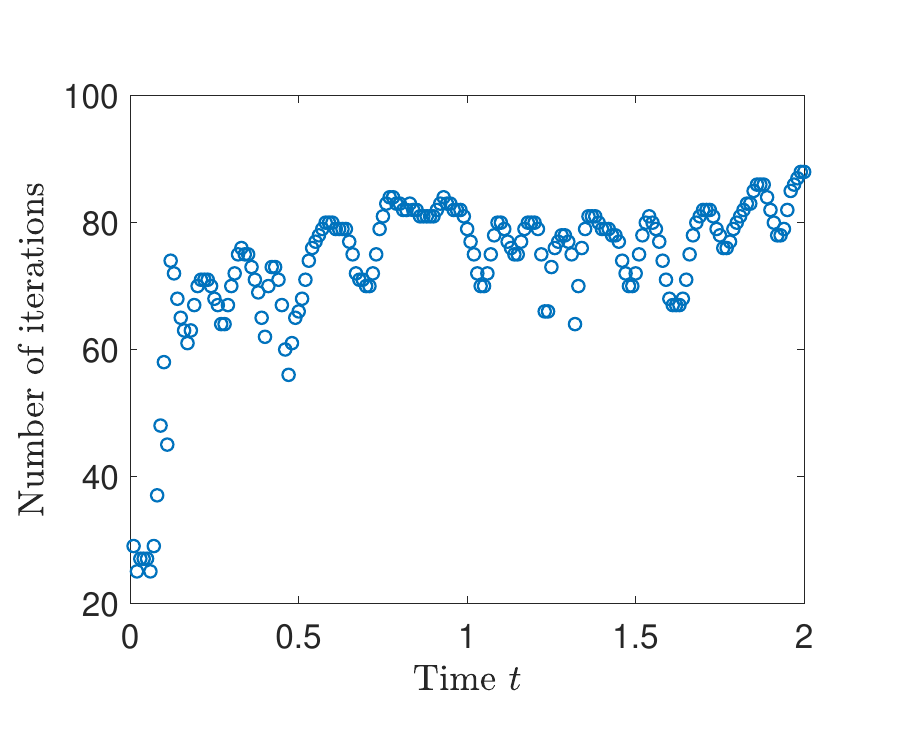}%
    \includegraphics[width=0.4\linewidth]{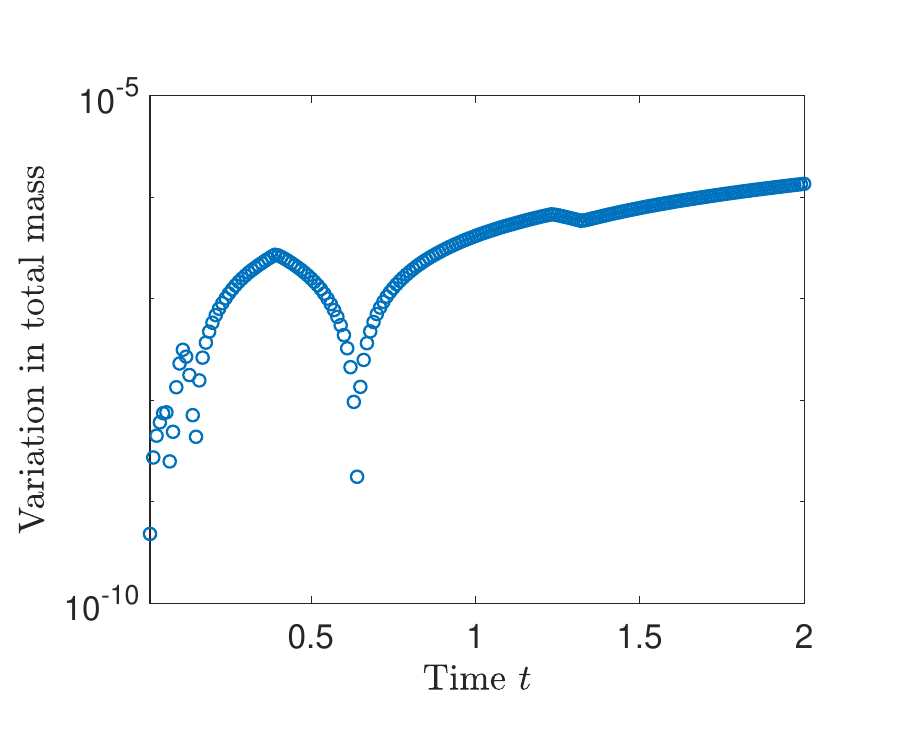}
\caption{(3D coarsening process)
Left: the number of iterations over time $t$, middle: the variation in total
mass over time $t$.}
\label{fig:other31}
\end{figure}

\begin{figure}[!t]
    \centering
    \includegraphics[width=0.4\linewidth]{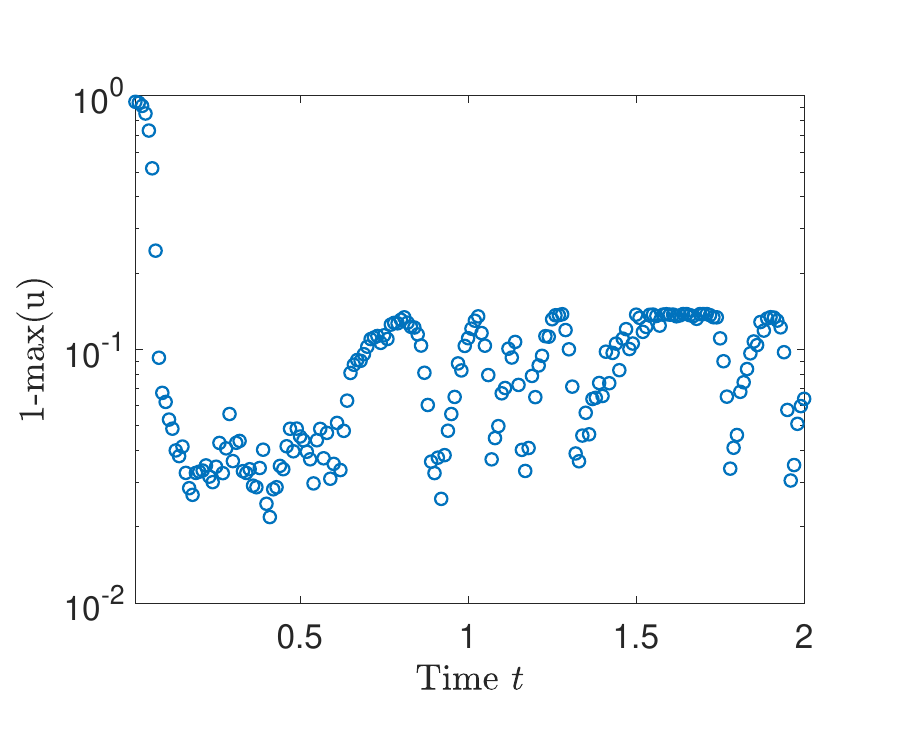}%
    \includegraphics[width=0.4\linewidth]{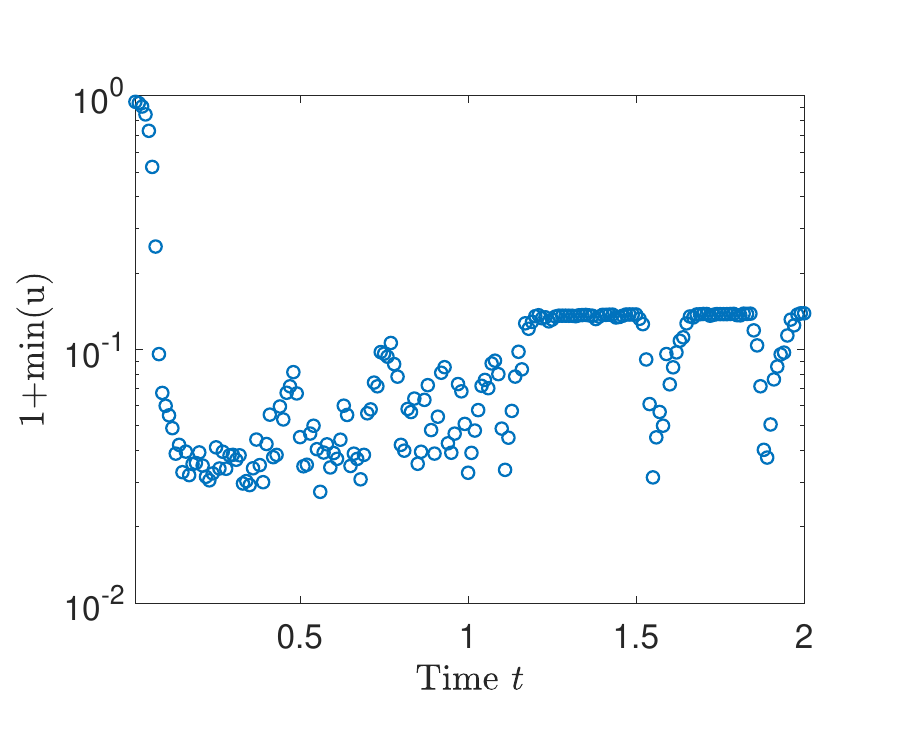}
\caption{(3D coarsening process)
$1-\max(u)$ and $1+\min(u)$ over time $t$.}
\label{fig:other34}
\end{figure}

\section{Concluding remarks}
\label{sec:conclusion}

In this paper, we introduce a novel iterative solver for the discrete system derived from the Cahn-Hilliard equation with the Flory-Huggins potential. By reformulating the original nonlinear system as an equivalent minimax optimization problem, we employ a new variant of the ADMM to address the resulting saddle-point minimax problem. We rigorously prove the convergence of our solver without imposing additional restrictions, such as conditions on the time step or the strong separability property. Furthermore, we analyze the bound-preserving characteristics and the evolution of the total mass. Numerical experiments, including coarsening process simulations in both two and three dimensions, support our theoretical findings.

As an initial study in this area, we focus on the finite difference method for spatial discretization under periodic boundary conditions. Nevertheless, the proposed solver can be readily extended to other spatial discretization techniques, such as the finite element method \cite{yuan2022second}, and to various boundary conditions, including Neumann and Dirichlet types. Applying this solver framework to other phase field models and to more complex scenarios, such as solar cell simulations \cite{wodo2012modeling}, is a promising direction for future work. Additionally, our iterative algorithm for nonlinear systems is currently tailored to the convex splitting numerical scheme. We aim to further reduce the number of iterations required by optimizing parameter choices and constructing effective initial guesses. Finally, inspired by the ideas presented in this work, we plan to investigate directly linearized discrete schemes to avoid the high computational cost of repeatedly solving nonlinear systems at each time step.

\bibliographystyle{amsplain}
\bibliography{article.bib}

\end{document}